\documentclass[a4paper,11pt,twoside,notitlepage,leqno]{article}
\usepackage{amsmath,amsthm,amsfonts,amscd,amssymb,calrsfs}
\usepackage[nottoc]{tocbibind}         
\input{xy}\xyoption{all}
\usepackage{hyperref}                  

\makeindex

\renewcommand{\ss}{\mathrm{ss}}

\newcommand{\bbN}{\mathbb{N}}

\newcommand{\bbQ}{\mathbb{Q}}

\newcommand{\bbu}{\boldsymbol{1}}

\newcommand{\mcA}{\mathcal{A}}

\newcommand{\mcB}{\mathcal{B}}

\newcommand{\mcC}{\mathcal{C}}

\newcommand{\mcS}{\mathcal{S}}
\newcommand{\mcT}{\mathcal{T}}

\DeclareMathOperator{\slg}{slg}

\newcommand{\bX}{\boldsymbol{X}}
\newcommand{\bY}{\boldsymbol{Y}}

\newcommand{\bZ}{\boldsymbol{Z}}

\newcommand{\Xss}{X^\mathrm{ss}}
\newcommand{\Yss}{Y^\mathrm{ss}}

\DeclareMathOperator{\id}{id}
\DeclareMathOperator{\coker}{coker}
\DeclareMathOperator{\End}{End}

\DeclareMathOperator{\modd}{mod}

\DeclareMathOperator{\Spec}{Spec}
\DeclareMathOperator{\Hom}{Hom}

\DeclareMathOperator{\iHom}{\mathbf{Hom}}
\DeclareMathOperator{\im}{im}

\DeclareMathOperator{\rad}{rad}
\DeclareMathOperator{\soc}{soc}

\DeclareMathOperator{\rk}{rk}
\DeclareMathOperator{\Gal}{Gal}

\DeclareMathOperator{\GL}{GL}

\DeclareMathOperator{\Aut}{Aut}

\DeclareMathOperator{\ind}{ind}

\DeclareMathOperator{\res}{res}

\DeclareMathOperator{\ev}{ev}

\DeclareMathOperator{\Stab}{Stab}

\newcommand{\isom}{\cong}
\newcommand{\llkurv}{(\!(}
\newcommand{\rrkurv}{)\!)}

\newcommand{\To}{\longrightarrow}

\newcommand{\Ksep}{K^{\mathrm{sep}}}

\newcommand{\Falg}{\overline{F}}

\newcommand{\bT}{\boldsymbol{T}}

\DeclareMathOperator{\Rep}{Rep}
\DeclareMathOperator{\Vect}{Vec}

\newcommand{\modE}{\text{$\modd$-$E$}}
\newcommand{\modEp}{\text{$\modd$-$E'$}}
\newcommand{\modEFp}{\text{$\modd$-$(E\otimes_FF')$}}

\swapnumbers                      
\numberwithin{equation}{subsection}

\theoremstyle{plain} 
\newtheorem{cor}[equation]{Corollary}
\newtheorem{lem}[equation]{Lemma}

\newtheorem{prop}[equation]{Proposition}
\newtheorem{thm}[equation]{Theorem} 

\theoremstyle{definition} 
\newtheorem{dfn}[equation]{Definition}

\newtheorem{exs}[equation]{Examples}

\theoremstyle{remark} 
\newtheorem{rem}[equation]{Remark}

\renewcommand{\emptyset}{\varnothing}

\mathchardef\ordinarycolon\mathcode`\:
\mathcode`\:=\string"8000
\begingroup \catcode`\:=\active
  \gdef:{\mathrel{\mathop\ordinarycolon}}
\endgroup

\newbox\mybox

\def\arrover#1{\mathrel{
\setbox\mybox=\hbox spread 1.4em{\hfil$\scriptstyle#1$\hfil}
\vbox{\offinterlineskip\copy\mybox \hbox
to\wd\mybox{\rightarrowfill}}}}

\def\larrover#1{\mathrel{
\setbox\mybox=\hbox spread 1.4em{\hfil$\scriptstyle#1$\hfil}
\vbox{\offinterlineskip\copy\mybox \hbox
to\wd\mybox{\leftarrowfill}}}}

\def\ontoover#1{\mathrel{
\setbox\mybox=\hbox spread 1.4em{\hfil$\scriptstyle#1$\hfil}
\vbox{\offinterlineskip\copy\mybox \hbox
to\wd\mybox{\rightarrowfill\hskip-2.8mm $\rightarrow$}}}}

\def\leftontoover#1{\mathrel{
\setbox\mybox=\hbox spread 1.4em{\hfil$\scriptstyle#1$\hfil}
\vbox{\offinterlineskip\copy\mybox \hbox
to\wd\mybox{$\leftarrow$\hskip-2.8mm \leftarrowfill}}}}

\def\into{\hookrightarrow}

\newdir^{ (}{{}*!/-3pt/\dir^{(}}
\newdir_{ (}{{}*!/-3pt/\dir_{(}}

\begin{document}
\author{Nicolas Stalder\footnote{Dept.\ of Mathematics, ETH Zurich, 8092 Zurich, Switzerland, nicolas@math.ethz.ch}}
\title{Scalar Extension of Abelian and Tannakian Categories}
\date{\today}
\maketitle
\begin{abstract}
We introduce and develop the notion of scalar extension for abelian categories. Given a field extension
$F'/F$, to every $F$-linear abelian category $\mcA$ satisfying a suitable finiteness
condition 
we associate an $F'$-linear abelian category $\mcA\otimes_FF'$ and an exact
$F$-linear functor $t:\:\mcA\to\mcA\otimes_FF'$.
This functor is universal among $F$-linear right exact functors with target an $F'$-linear
abelian category.

We discuss various basic properties of this concept, among others compatibilities with
multilinear endofunctors such as tensor products, and the permanence
of favourable properties of the functors and categories involved.
We obtain the notion of scalar extension for Tannakian categories,
which allows us to deduce consequences for the algebraic monodromy groups of
Tannakian categories.
\end{abstract}

\tableofcontents


\section*{Introduction}
\addcontentsline{toc}{section}{Introduction}

This article initiates a study of ``scalar extension''
of abelian categories, in the case where the scalars
are fields. We understand this to mean that to a field extension $F'/F$ 
and an $F$-linear abelian category $\mcA$
we wish to associate an $F'$-linear abelian category $\mcA\otimes_FF'$
and an $F$-linear exact functor \[t:\:\mcA\To\mcA\otimes_FF'\]which is in
a certain sense ``universal'' among certain $F$-linear functors with values
in an $F'$-linear abelian category.

We construct $\mcA\otimes_FF'$ and $t$ in Subsections \ref{ss:thecategory} and \ref{ss:thefunctor},
under the assumption that $\mcA$ is \emph{$F$-finite}: All objects
have finite length, and all endomorphism algebras are finite $F$-dimensional  (Definition \ref{dfn:Ffinite}).
In the case of Tannakian categories, our construction has been used before, we know
of the instances \cite{DeM82} and \cite{Mil92}.

What is original in our approach is to characterise
this construction by finding its universal property, applicable to all abelian categories.
Namely, every \emph{right}-exact
$F$-linear functor $V:\:\mcA\to\mcB$ with target an $F'$-linear abelian category
has a \emph{right}-exact $F'$-linear ``extension'' $V':\:\mcA\otimes_FF'\to\mcB$ which
is unique ``up to unique isomorphism'':
\begin{equation*}
\text{($\star$)}\qquad\qquad\vcenter{\xymatrix{
\mcA \ar[rr]^{t}\ar[dr]_V && \mcA\otimes_FF' \ar@{.>}[dl]^{V'}\\
&\mcB}}\end{equation*}
This is the content of Subsection $1.4$. We refer
to Theorem \ref{thm:univpropscalex} for the precise formulation of
this universal property, where the underlying $2$-categorical nonsense
is formulated in precise, down-to-earth terms.

Examples of this process are plentiful, and show that our abstract notion
of scalar extension coincides with what intuition suggests.
For instance, if $E$ is a finite-dimensional $F$-algebra and $\mcA$ is the category of finite $F$-dimensional $E$-modules,
then $\mcA\otimes_FF'$ is the category of finite $F'$-dimensional $(F'\otimes_FE)$-modules.

We hasten to add that an exact functor on $\mcA$ need \emph{not} extend to an exact
functor on $\mcA\otimes_FF'$. 
However, by categorical
nonsense the category $(\mcA^\mathrm{op}\otimes_FF')^\mathrm{op}$ has the univeral
property that left exact functors do extend. So a
possible direction of further research might be the following question:
Under which conditions do the categories $(\mcA^\mathrm{op})\otimes_FF'$
and $(\mcA\otimes_FF')^\mathrm{op}$ coincide?

Throughout the article, we shall systematically disregard set-theoretic
difficulties. For hints towards a solution of these, and general categorical background,
we refer to \cite{Kas06}.

\subsection*{Motivation and Overview}
The content of this article is not meant to be ``l'art pour l'art''.
My main motivation for its presentation are two applications to Tannakian duality,
which I use in my article \cite{Sta08}.
For the first, recall that to a Tannakian category $\mcT$ over $F$ with fibre functor
$\omega$ over $F'$ there is associated a linear algebraic group $G_\omega(\mcT)$
over $F'$, the \emph{algebraic monodromy group} of $\mcT$ with respect to
$\omega$. It turns out -- Theorem \ref{thm:nonneutralTannaka} --
that the functor $\omega'$ induced by $\omega$ using
the universal property identifies $\mcT\otimes_FF'$ with
the category of finite-dimensional representations of $G_\omega(\mcT)$ over $F'$.
In this way, we obtain as our first application a weak form
of non-neutral Tannakian duality, which uses only the input of neutral Tannakian duality.

To prove this fact, we must first show that $\mcT\otimes_FF'$ carries a tensor
product, and that ``everything is compatible'' with tensor products. For
this, in Subsection \ref{ss:scalextoftenscats} we will consider more generally a multilinear functor
$\mcA^{\times n}\to\mcA$ and study the induced multilinear functor
$(\mcA\otimes_FF')^{\times n}\to\mcA\otimes_FF'$. 
Since in tensor categories with duals right exact functors are automatically
exact -- Lemma \ref{lem:rightexactisexactforrigids} -- the proof of our
first application ensues rather easily.

The second application is a partial answer to the following
question: In diagram ($\star$), 
under which conditions are favourable properties
of $V'$ equivalent to corresponding ``relatively'' favourable
properties of $V$? An example has been given above, the question of
being exact. The two others we focus on are the following:
When is $V'$ fully faithful? And when is the essential image of $V'$
closed under subquotients? Taken together, we ask: When, in
terms of properties of $V$ and the field extension $F'/F$, is
$V'$ an equivalence of categories?

The relative version of being fully faithful is to be \emph{$F'/F$-fully faithful},
a categorical version of the Tate conjecture on homomorphisms
in algebraic geometry, see Definition \ref{dfn:FpFfullyfia}. We prove that $t$
is $F'/F$-fully faithful in Proposition \ref{lem:tisFpFff}.
For tensor categories with duals,
we prove that $V$ is $F'/F$-fully faithful if and only if
$V'$ is fully faithful in Subsection \ref{ss:permofrelff}.

I have not achieved a full clarification of what the relative version
of the essential image being closed under subquotients
is.
As a kludge, in the special case of separable field
extensions and the context of tensor categories
with duals, we study
functors which map semisimple objects to semisimple objects,
we call this property \emph{semisimple on objects}.
If $F'=F$,  this property is equivalent to
the essential image being closed under subquotients
for exact fully faithful
functors by Proposition \ref{prop:ffssisessimsubquotclosed}.
If $F'/F$ is separable,
we prove that $t$ is semisimple on objects in Proposition \ref{prop:FpFsepthentsemsimple},
and that $V$ is semisimple on objects if and only if $V'$ is in
Proposition \ref{prop:permofsemisimpli}.

Our second application is then a partial answer to the
question of ``recognising induced equivalences of categories''.
It is developed in Subsection \ref{ss:inducedequivalences}, and states the following.
Let $F'/F$ be a separable field extension, $\mcT$ a Tannakian category
over $F$, $\mcT'$ a Tannakian category over $F'$ and
$V:\:\mcT\to\mcT'$ an $F$-linear exact tensor functor. If $V$ is $F'/F$-fully faithful
and semisimple on objects, then the induced functor
\[V':\:\mcT\otimes_FF'\To\llkurv V\mcT\rrkurv_{\otimes}\]
is an equivalence of Tannakian categories, where $\llkurv V\mcT\rrkurv_{\otimes}$
denotes the strictly full Tannakian subcategory of $\mcT'$ generated
by the image of $V$.

\subsection*{Motivic Monodromy Groups}

Here is an example of how our second application may be put to use.
Let $\mcA$ be the $\bbQ$-linear abelian subcategory of 
the category of pure Grothendieck
motives -- up to isogeny and numerical equivalence --
generated by abelian varieties over a given number field $K$.
By \cite{Jan92}, it is a Tannakian category. Choose a prime number $\ell$,
and let $\mcB$ denote the category of finite-dimensional continuous representations
of the absolute Galois group $\Gamma:=\Gal(\Ksep/K)$ of $K$ over $\bbQ_{\ell}$.
This is obviously a Tannakian category, and it is known -- Proposition \ref{prop:mondrofcontreps} --
that the algebraic monodromy
group of its strictly full Tannakian subcategory $\llkurv V\rrkurv_{\otimes}$
generated by a given Galois representation $V$ with respect to the forgetful
functor may be identified with Zariski closure of the image of $\Gamma$
in $\GL(V)(\bbQ_{\ell})$. By \cite{Fal83}, the functor $V_\ell$ of rational Tate modules
is indeed $\bbQ_\ell$-fully faithful (Tate's conjecture!) and semisimple
on objects. Since all objects of $\mcA$ are semisimple by Poincar\'e reducibility and \cite{Jan92}, this
latter property means that all rational Tate modules of abelian varieties are semisimple,
and is hence a special case of the Grothendieck-Serre conjecture on
etale cohomology groups of pure motives.
Therefore, our theorem allows to conclude that the algebraic monodromy group of an
abelian variety over a number field -- its ``motivic'' monodromy group -- 
coincides with the Zariski-closure of the image of Galois.
Our application of Theorem \ref{thm:mainthm2} in \cite{Sta08} is an anologue
of this example, with Anderson $A$-motives replacing abelian
varieties, and the main result of \cite{Sta08}, as advertised
in its title, replacing \cite{Fal83}.

\subsection*{Acknowledgments}
This article as well as \cite{Sta08} are developments of my Ph.D. thesis.
It is my pleasure to thank Richard Pink for his guidance during my doctoral
studies.

\section{Abelian Categories}

\subsection{Properties of Algebras, Categories and Functors}\label{ss:propsofcatsandfuncts}
In this subsection, we collect several
algebraic and categorical notions for later reference.
Let $F$ be a field.
Recall that a category is
\emph{$F$-linear} if all $\Hom$-sets are endowed
with $F$-vector space structures in such a
way that composition of homomorphisms is $F$-bilinear.

\begin{dfn}\label{dfn:Ffinite}
Let $\mcA$ be an abelian category.
\begin{enumerate}
  \item $\mcA$ is \emph{finite} if all
  objects are of finite length.
  \item Assume that $\mcA$ is $F$-linear. Then
  $\mcA$ is \emph{$F$-finite} if it is finite
  and the endomorphism algebra of each object is finite $F$-dimensional.
\end{enumerate}
\end{dfn}

\begin{dfn}\label{dfn:FpFfullyfia}
Let $F'/F$ be a field extension.
Consider an $F$-linear category $\mcC$ and
an $F'$-linear category $\mcC'$. An $F$-linear
functor $V:\:\mcC\to\mcC'$ is \emph{$F'/F$-fully faithful}
if the induced homomorphism
\[F'\otimes_F\Hom_{\mcC}(X,Y)\to\Hom_{\mcC'}\big(V(X),V(Y)\big),\quad f'\otimes h\mapsto f'\cdot V(h)\]
is an isomorphism for all objects $X,Y$ of $\mcC$.
\end{dfn}
More loosely speaking, we might say that an $F'/F$-fully faithful
functor is \emph{relatively fully faithful} if the field
extension $F'/F$ is clear from the context.

\begin{dfn}\label{dfn:semisimpleonobjects}
Let $\mcA$ and $\mcB$ be abelian categories.
An exact functor $V:\:\mcA\to\mcB$ is \emph{semisimple on objects}
if it maps all semisimple objects of $\mcA$ to semisimple
objects of $\mcB$.
\end{dfn}

Let us consider a consequence of the juxtaposition
of the two properties ``$F'/F$-fully faithful'' and ``semisimple on objects''
in the special case $F'=F$.

\begin{prop}\label{prop:ffssisessimsubquotclosed}
Let $\mcA$ be a finite $F$-linear abelian category, $\mcB$ an $F$-linear abelian
category, and $V:\:\mcA\To\mcB$ an $F$-linear, exact, fully faithful functor semisimple on objects.
Then the essential image of $V$ is closed under subquotients in $\mcB$.
\end{prop}
\begin{proof}
By symmetry, it is enough to show that the essential image of $V$
is closed under subobjects in $\mcB$. So let $A$ be an object of $\mcA$ and
$B$ an arbitrary subobject of $V(A)$ in $\mcB$. We must show that $B\isom V(A_0)$ for some
object $A_0$ of $\mcA$. We proceed by induction on $\ell:=\lg(B)$,
the length of a composition series of $B$.
The case $\ell=0$ is trivial. If $\ell=1$, but $A$ is not simple, choose a short exact sequence
\[0\To A'\To A\To A''\To 0\]
with nonzero objects $A',A''$ of $\mcA$. Consider the composite
homomorphism $h:\:B\to V(A)\to V(A'')$. There are two possibilities:
\begin{enumerate}
  \item ``$h\neq 0$'': In this case, $B$ is a subobject of $V(A'')$ since $B$ is simple, and $\lg(A'')<\lg(A)$.
  \item ``$h=0$'': In this case, $B$ is a subobject of $V(A')$, and $\lg(A')<\lg(A)$.
\end{enumerate}
Since $\lg(A)<\infty$, by repeating this process with $A'$ or $A''$ instead of $A$, depending on
which case we arrive at, we find a simple object $A_1$ of $\mcA$ such that $B$ is a subobject of $V(A_1)$.
Now $V(A_1)$ is semisimple since $V$ is semisimple on objects, so $B$
is a quotient object of $V(A_1)$. Since $V$ is fully faithful,
the composite homomorphism $g:\:V(A_1)\to B\to V(A_1)$ is of the form $V(f)$ for some homomorphism
$f\in\End(A_1)$. Set $A_0:=\im(f)$. Since $V$ is exact, we see that $B=\im(g)\isom V(A_0)$, as required.

We turn to the case $\ell=\lg(B)>1$. Choose a short exact sequence
\[0\To B'\To B\To B''\To 0\]
with nonzero objects $B',B''$ of $\mcB$. By induction hypothesis, $B'\isom V(A')$
and $B''\isom V(A'')$ for objects $A',A''$ of $\mcA$. Consider the induced commutative
diagram with exact rows
\[\xymatrix{
0 \ar[r] & V(A') \ar[r] \ar@{=}[d] & B \ar[r] \ar@{^(->}[d] & V(A'') \ar[r] \ar@{.>}[d] & 0 \\
0 \ar[r] & V(A') \ar[r]            & V(A) \ar[r]           & V(A/A') \ar[r]             & 0,
}\]
using the fact that $V$ is exact.
The dotted vertical arrow is of the form $g=V(f)$ since $V$ is fully faithful.
Set $A''':=\coker(f)$. Consider the composite homomorphism
\[h:\:V(A)\to V(A/A')\arrover{g}V(A'''),\]
where we again use the fact that $V$ is exact.
The Snake Lemma implies that $B=\ker(h)$. Since $V$ is fully faithful,
$h$ is of the form $V(f')$ for some homomorphism $f':\:A\to A'''$.
Set $A_0:=\ker(f')$. Since $V$ is exact, we see that
$B=\ker(h)\isom V(A_0)$, as required.
\end{proof}

In the situation of Definition \ref{dfn:semisimpleonobjects},
if $F'/F$ is a field extension, $\mcA$ is $F$-linear and
$\mcB$ is $F'$-linear, experience tells us not to expect
an exact functor $V:\:\mcA\to\mcB$ to be semisimple on
objects in the absence of separability assumptions. Hence we recall the
definition of separability for $F$-algebras, which
extends the usual definition of separability for finite
field extensions.

\begin{dfn}\label{dfn:sepfieldext}
A field extension $F'/F$ is \emph{separable} if
the ring $\overline{F}\otimes_FF'$ contains no nilpotent elements,
where $\overline{F}$ denotes an algebraic closure of $F$.
\end{dfn}

\begin{dfn}\label{dfn:sepalgebra}
A finite-dimensional semisimple $F$-algebra $E$ is \emph{separable}
if the center of each simple factor is a separable field
extension of $F$.
\end{dfn}

\begin{prop}\label{prop:separablealgebras}
Let $E$ be a finite-dimensional semisimple $F$-algebra, and
consider a field extension $F'/F$. If either
\begin{enumerate}
  \item $F'/F$ is a separable field extension, or
  \item $E$ is a separable $F$-algebra,
\end{enumerate}
then $F'\otimes_FE$ is a semisimple $F'$-algebra.
\end{prop}
\begin{proof}
(a): \cite[\S7, no. 3, Corollaire 1 to Proposition 3(b)]{Bou58}.

(b): \cite[\S7, no. 5, Proposition 6 and Corollary to Proposition 7]{Bou58}.
\end{proof}

At the end of Subsection \ref{ss:thefunctor},
we will need the property given in
the following definition, which in contrast to
semisimplicity is invariant under field extensions.
For more information, see \cite{Lam99}.

\begin{dfn}
A finite-dimensional $F$-algebra $E$ is \emph{Frobenius}
if there exists an isomorphism $E\isom\Hom_F(E,F)$ of
right $E$-modules.
\end{dfn}

\begin{prop}\label{prop:frobeniusalgebras}
Let $E$ be a finite-dimensional $F$-algebra.
\begin{enumerate}
  \item If $E$ is a semisimple $F$-algebra, then $E$ is Frobenius.
\end{enumerate}
Assume that $E$ is Frobenius.
\begin{enumerate}
\setcounter{enumi}{1}
  \item For every field extension $F'/F$ the $F'$-algebra
  $F'\otimes_FE$ is Frobenius.
  \item We  have $\soc(E)\isom E/\rad(E)$ as right $E$-modules,
  where $\soc(E)$ denotes the maximal semisimple right $E$-submodule
  of $E$, and $\rad(E)$ denotes the maximal semisimple right $E$-module
  quotient of $E$.
\end{enumerate}
\end{prop}
\begin{proof}
Set $E^\vee:=\Hom_F(E,F)$, considered as a right $E$-module
using the left $E$-module structure of $E$.

(a): By additivity, we may assume that $E$ is a simple $F$-algebra.
Then, up to isomorphism, there exists only one simple
right $E$-module, and so the isomorphism class of a right $E$-module
is determined by its dimension over $F$. Since $\dim_FE=\dim_FE^\vee$,
it follows that $E$ and $E^\vee$ are isomorphic.

(b): By assumption, $E\isom E^\vee$, and hence
$F'\otimes_FE\isom F'\otimes_F\Hom_F(E,F)\isom\Hom_{F'}(F'\otimes_FE,F')$
as claimed.

(c): By duality, $\soc(E^\vee)=\big(E/\rad(E)\big)^\vee$. Since $E$ is
a Frobenius $F$-algebra, we obtain an induced isomorphism $\soc(E)\isom\soc(E^\vee)$.
Now $E/\rad(E)$ is a semisimple $F$-algebra, so item (a) implies that $(E/\rad(E))^\vee\isom E/\rad(E)$
as right $E/\rad(E)$-modules, and thus as right $E$-modules. Taken together,
we obtain a composite isomorphism
\[\soc(E)\isom\soc(E^\vee)=(E/\rad(E))^\vee\isom E/\rad(E)\]
of right $E$-modules, as claimed.
\end{proof}

\subsection{The Category $\mcA\otimes_FF'$}\label{ss:thecategory}

Let $F$ be a field, and consider an $F$-linear abelian category $\mcA$.

\begin{dfn}\label{dfn:indcat}
An \emph{ind-object of $\mcA$} is a filtered direct system $(X_i)_{i\in I}$ of
objects of $\mcA$. A \emph{homomorphism} of two given ind-objects $(X_i)_{i\in I}$ and $(Y_j)_{j\in J}$
of $\mcA$ is an element of $\varprojlim_i\varinjlim_j\Hom_{\mcA}(X_i,Y_j)$. We obtain
the $F$-linear abelian category $\ind\mcA$ of ind-objects of $\mcA$.
\end{dfn}

We have a natural functor $\mcA\to\ind\mcA$, mapping an object $X$ of $\mcA$
to the object $(X_i)_{i\in I_\emptyset}$
given by $I_\emptyset:=\{\emptyset\}$ and $X_\emptyset:=X$.
It is $F$-linear, exact and fully faithful. Abusing notation,
we identify $\mcA$ with its essential image in $\ind\mcA$ under this natural
functor.

Recall that $\mcA$ is finite if all of its objects have finite length.

\begin{lem}\label{lem:indcatgen}
Let $\mcA$ be finite.
\begin{enumerate}
  \item $\mcA$ is closed under subquotients in $\ind\mcA$.
  \item Every object of $\ind\mcA$ is a union of subobjects in $\mcA$.
\end{enumerate}
\end{lem}
\begin{proof}
\cite[\S4.1 and Lemme 4.2.1]{Del87}.
\end{proof}

Let $F'/F$ be a field extension.

\begin{dfn}
An \emph{$F'$-module} in $\ind\mcA$ is an ind-object of $\mcA$ together with
an $F$-linear ring homomorphism $\phi:\:F'\to\End_{\ind\mcA}(X)$. A \emph{homomorphism}
of $F'$-modules is a homomorphism of ind-objects which commutes with the respective
actions of $F'$. We obtain the $F'$-linear abelian category $(\ind\mcA)_{F'}$ of
$F'$-modules in $\ind\mcA$.
\end{dfn}

Recall that $\mcA$ is $F$-finite if it is finite and the endomorphism
algebra of each object is finite $F$-dimensional.

\begin{dfn}
Let $\mcA$ be $F$-finite. The \emph{scalar extension} of $\mcA$
from $F$ to $F'$ is the full subcategory $\mcA\otimes_FF'$ of $(\ind\mcA)_{F'}$ consisting
of all $F'$-modules of finite length. It is $F'$-linear, abelian and finite.
\end{dfn}

\begin{rem}
In the next section, we will see that $\mcA\otimes_FF'$ is $F'$-finite.
\end{rem}

\begin{exs}\label{ex:scalarextcats}
\begin{enumerate}
  \item If $E$ is a finite-dimensional $F$-algebra and $\mcA$ is the category of finite $F$-dimensional left $E$-modules, then $\mcA\otimes_FF'$
  is the category of finite $F'$-dimensional left $(F'\otimes_FE)$-modules.
  \item If $G$ is an affine group scheme over $F$ and $\mcA$ is the category $\Rep_F(G)$ of finite-dimensional representations
  of $G$ over $F$, then $\mcA\otimes_FF'$ is the category $\Rep_{F'}(G_{F'})$
  of finite-dimensional representations of $G_{F'}:=G\times_{\Spec(F)}\Spec(F')$ over $F'$.
  For a proof, we refer to \cite{Del87}.
\end{enumerate}
\end{exs}

\subsection{The Functor $t:\:\mcA\to\mcA\otimes_FF'$}\label{ss:thefunctor}

Let $F$ be a field, and consider an $F$-finite $F$-linear abelian category $\mcA$.

\begin{dfn}\label{dfn:extscalfield}
Consider an object $X$ of $\ind\mcA$, an $F$-subalgebra $E\subset\End_{\ind\mcA}(X)$, and
a free right $E$-module $M$. The \emph{external tensor product} $M\otimes_EX$ of $M$ with $X$ over $E$
is the object of $\ind\mcA$ representing the functor from $\ind\mcA$ to left $E$-modules given by
$Y\mapsto\Hom_E(M,\Hom_{\ind\mcA}(X,Y))$. In other words, we require a natural
isomorphism
\[\Hom_{\ind\mcA}(M\otimes_EX,Y)\arrover{\isom}\Hom_E\big(M,\Hom_{\ind\mcA}(X,Y)\big).\]
Note that $M\otimes_EX$ exists, it is a direct sum of $\rk_E(M)$ copies of $X$.
\end{dfn}

The external tensor product is an exact $F$-linear functor in its first variable if $X$ and $E$ are fixed, and
in its second variable if $E=F$ and $M$ is fixed.

\begin{rem}\label{rem:up2ofscalext2}
Consider the situation of Definition \ref{dfn:extscalfield}. If $M$ is a free right $E$-module
of \emph{finite} rank, then $M\otimes_EX$ also represents 
the functor $Z\mapsto M\otimes_E\Hom_{\ind\mcA}(X,Z)$ on $\ind\mcA$,
so one has a natural isomorphism
\[M\otimes_E\Hom_{\ind\mcA}(Z,X)\arrover{\isom}\Hom_{\ind\mcA}(Z,M\otimes_EX).\]
\end{rem}

Let $F'/F$ be a field extension. For every object $X$ of $\ind\mcA$, the external tensor product $F'\otimes_FX$
has a natural $F'$-module structure, using the action of $F'$ on itself by multiplication $\mu$.
We obtain an exact $F$-linear functor
\begin{equation}\label{eqn:goingtoindfp}t=t_{F'/F}:\:\ind\mcA\To (\ind\mcA)_{F'},\quad X\mapsto (F'\otimes_FX,\mu\otimes\id).\end{equation}
We also let $t$ denote its restriction to $\mcA$.

\begin{prop}\label{prop:tisleftadjtoforge}
For every object $X$ of $\ind\mcA$ and $\bY=(Y,\psi)$ of $(\ind\mcA)_{F'}$, the restriction
homomorphism
\[\Hom_{(\ind\mcA)_{F'}}\left(t(X),\bY\right)\To\Hom_{\ind\mcA}(X,Y)\]
is an isomorphism. In other words, $t$
is left adjoint to the forgetful functor from $F'$-modules in $\ind\mcA$ to $\ind\mcA$ itself.
\end{prop}
\begin{proof}
We construct an inverse $e$ to the restriction homomorphism.
Given a homomorphism $h:\:X\to Y$, the induced homomorphism
\[F'\to\Hom(X,Y),\quad f'\mapsto\left(X\arrover{h} Y\arrover{\psi(f')}Y\right)\]
corresponds to a unique homomorphism $e(h):\:F'\otimes_FX\to Y$ by the definition
of $F'\otimes_FX$. By construction, $e(h)$ is a homomorphism of $F'$-modules.
\end{proof}

\begin{rem}\label{rem:tXbXcan}
Given an $F'$-module $\bX=(X,\phi)$ in $\ind\mcA$, Proposition \ref{prop:tisleftadjtoforge}
implies that there exists a natural homomorphism $t(X)\to\bX$ corresponding to $\id_X$
via $\phi$. Note that this homomorphism is surjective.
\end{rem}

Recall the notion $F'/F$-fully faithful, introduced in Definition \ref{dfn:FpFfullyfia}.

\begin{prop}\label{lem:tisFpFff}
The functor $t:\:\mcA\to(\ind\mcA)_{F'}$ is $F'/F$-fully faithful.
\end{prop}
\begin{proof}
We must show that for all objects $X,Y$ of $\mcA$
the natural homomorphism
\[F'\otimes_F\Hom_{\ind\mcA}(X,Y)\To\Hom_{(\ind\mcA)_{F'}}\big(t(X),t(Y)\big)\]
is an isomorphism. By Proposition \ref{prop:tisleftadjtoforge}, the target of this homomorphism
coincides with $\Hom_{\ind\mcA}(X,F'\otimes_FY)$, so we must show that
the natural homomorphism
\[F'\otimes_F\Hom_{\ind\mcA}(X,Y)\To\Hom_{\ind\mcA}(X,F'\otimes_FY)\]
is an isomorphism.

\emph{Injectivity:} Given a non-zero element $h'$ of $F'\otimes_F\Hom(X,Y)$, there
exists a finite $F$-dimensional subspace $V\subset F'$ such that $h'$
arises from an element $\widetilde{h'}$ of $V\otimes_F\Hom(X,Y)$. By Remark \ref{rem:up2ofscalext2},
we have a natural isomorphism $V\otimes_F\Hom(X,Y)\isom\Hom(X,V\otimes_FY)$. Now the
commutative diagram (disregard $h$ and $\widetilde{h}$ for the moment)
\begin{equation}\label{eqn:somecommuttdiagg}\vcenter{\xymatrix{\widetilde{h'}\in V\otimes_F\Hom(X,Y) \ar[r]^\isom \ar@{^(->}[d] & \Hom(X,V\otimes_FY)\ni\widetilde{h} \ar@{^(->}[d] \\
            h'\in F'\otimes_F\Hom(X,Y) \ar[r]       & \Hom(X,F'\otimes_FY)\ni h
}}\end{equation}
implies that $h'$ is mapped to a non-zero element $h$ of $\Hom(X,F'\otimes_FY)$.

\emph{Surjectivity:} Consider an element $h$ of $\Hom(X,F'\otimes_FY)$. Since
$\mcA$ is finite the object $X$ has finite length, so the image $\im(h)$ of $h$ is of finite length as well.
The object $F'\otimes_FY$
is the union over all finite $F$-dimensional subspaces $W\subset F'$
of its subobjects $W\otimes_FY$. It follows that $\im(h)\subset V\otimes_FY$ for
some finite $F$-dimensional vector subspace $V\subset F'$.

Therefore, $h$ arises from an element $\widetilde{h}$ of $\Hom(X,V\otimes_FY)$. Now the commutative
diagram (\ref{eqn:somecommuttdiagg}) shows that $h$ is the image of an element $h'$
of $F'\otimes_F\Hom(X,Y)$.
\end{proof}

\begin{rem}
If $\mcA$ is not finite, then $t$ need not be $F'/F$-fully faithful.
Here is a counter-example: Set $F:=\bbQ$ and let $\mcA$ be the category
of all $\bbQ$-vector spaces. Consider $X:=\bigoplus_{j\in\bbN}\bbQ$, $Y:=\bbQ$
and $F':=\overline{\bbQ}$, an algebraic closure of of $\bbQ$.
As $\bbQ$-vector space $F'$ is isomorphic to $\bigoplus_{i\in\bbN}\bbQ$.
Then the homomorphism
\[F'\otimes_F\Hom(X,Y)\to\Hom(X,F'\otimes_FY)\]
is not surjective.
Indeed, we have $F'\otimes_F\Hom(X,Y)\isom\bigoplus_{i\in\bbN}\prod_{j\in\bbN}\bbQ$,
whereas $\Hom(X,F'\otimes_FY)\isom\prod_{j\in\bbN}\bigoplus_{i\in\bbN}\bbQ$.
The latter strictly contains the former.
\end{rem}

Next, we wish to show that the image of $\mcA$ under $t$ lies in $\mcA\otimes_FF'$.
For this, we study how simple objects of $\mcA$ ``split up'' under $t$.

\begin{dfn}
Let $S$ be a simple object of $\ind\mcA$. An object of $\ind\mcA$ is \emph{$S$-isotypic}
if it is a direct sum of copies of $S$.
\end{dfn}

\begin{lem}\label{lem:EmodXiso}
Let $S$ be a simple object of $\ind\mcA$ and set $E:=\End_{\ind\mcA}(S)$. The functor $(-)\otimes_ES$
is an equivalence of $F$-linear abelian categories between the category of (free) right $E$-modules and the
full subcategory of $S$-isotypic objects of $\ind\mcA$.
\end{lem}
\begin{proof}
A quasi-inverse functor is given by $\Hom_{\ind\mcA}(S,-)$.
\end{proof}

We turn to a construction. Let $S$ be a simple object of $\mcA$ and
set $E:=\End(S)$. Since $S$ is simple, $E$ is a skew field and all
right $E$-modules are free.
Let $\modE$ denote the $F$-finite $F$-linear abelian category of finite-dimensional $E$-modules,
and $\llkurv S\rrkurv$ the $F$-finite $F$-linear abelian category of $S$-isotypic
objects in $\mcA$. Restricting the statement of Lemma \ref{prop:ridealsubob}
to objects of finite length, we obtain an equivalence of $F$-finite $F$-linear
abelian categories $(-)\otimes_ES:\:\modE\arrover{\isom}\llkurv S\rrkurv$.
Let $F'/F$ be a field extension. We obtain an induced equivalence of finite $F'$-linear abelian categories
\[(\modE)\otimes_FF'\arrover{\isom}\llkurv S\rrkurv\otimes_FF'.\]
As in Example \ref{ex:scalarextcats}(a), we have $\modEFp\isom (\modE)\otimes_FF'$.
On the other hand, the inclusion $\llkurv S\rrkurv\subset\mcA$ induces
a natural fully faithful $F$-linear exact functor $\llkurv S\rrkurv\otimes_FF'\subset \mcA\otimes_FF'$.
Setting $E':=E\otimes_FF'$ we obtain a fully faithful $F'$-linear exact functor
\[\modEp\isom (\modE)\otimes_FF'\arrover{\isom}\llkurv S\rrkurv\otimes_FF'\subset \mcA\otimes_FF',\]
which we denote as $(-)\otimes_{E'}t(S)$.

\begin{prop}\label{prop:ridealsubob}
Let $S$ be a simple object of $\mcA$, set $E:=\End_{\mcA}(S)$
and $E':=F'\otimes_FE$. The functor $(-)\otimes_{E'}t(S)$
gives rise to an inclusion preserving bijection
from the set of right ideals of $E'$ to the set
of subobjects of $t(S)$ in $(\ind\mcA)_{F'}$.
\end{prop}
\begin{proof}
Recall that $E$ is a skew field and all right $E$-modules are free.
Set $S':=F'\otimes_FS$ and note that $S'$ is an $S$-isotypic object of $\ind\mcA$.

Consider the following diagram of lattices:
\[
\xymatrix{
\left\{\text{\begin{tabular}{c}right $E$-\\submodules of $E'$\end{tabular}}\right\} \ar@{->}^{(-)\otimes_ES}[r] & \left\{\text{\begin{tabular}{c}$S$-isotypic\\subobjects of $S'$\end{tabular}}\right\}\\
\left\{\text{\begin{tabular}{c}$F'$-stable right \\ $E$-submodules of $E'$\end{tabular}}\right\} \ar@{->}[r] & \left\{\text{\begin{tabular}{c}$F'$-stable $S$-isotypic \\ subobjects of $S'$\end{tabular}}\right\} \\
\left\{\text{ right ideals of $E'$ }\right\} \ar@{->}[r]^{(-)\otimes_{E'}t(S)} \ar@{=}[u]  & \left\{\text{ subobjects of $t(S)$ }\right\} \ar@{=}[u]
}
\]
The upper row is a bijection by Lemma \ref{lem:EmodXiso}, and it
preserves inclusions by construction. The second row corresponds to
the $F'$-stable objects in the upper row, using the action of
$F'$ on $E'$ and $S'$, respectively. Since the bijection in the first
row is natural, it induces a bijection in the second row. 
Finally, we may clearly identify the objects of the second row with
the vertically corresponding objects of the third row.
Unraveling the definition of $(-)\otimes_{E'}t(S)$ we see
that the resulting diagram commutes, so we have proven our claim.
\end{proof}

\begin{cor}
The image of $t:\:\mcA\to(\ind\mcA)_{F'}$ is contained in $\mcA\otimes_FF'$
and thus we obtain a well-defined functor $t:\:\mcA\to\mcA\otimes_FF'$.
\end{cor}
\begin{proof}
We must show that $t(X)$ has finite length as $F'$-module for every object $X$ of $\mcA$. Since $\mcA$ is finite
and $t$ is exact, we may assume that $X=:S$ is simple. Since $\mcA$ is $F$-finite, the endomorphism ring $E:=\End_{\mcA}(S)$ is
finite $F$-dimensional. It follows that $F'\otimes_FE$ has finite length as a right module
over itself, and so $t(S)$ has finite length by Proposition \ref{prop:ridealsubob}.
\end{proof}

Summing up what we have achieved so far, to our $F$-finite abelian category $\mcA$ 
we have associated a finite $F$-linear abelian category $\mcA\otimes_FF'$
and an $F$-linear exact $F'/F$-fully faithful functor $t:\:\mcA\to\mcA\otimes_FF'$.

The remainder of this subsection is devoted to showing that $\mcA\otimes_FF'$ is
not only finite, but indeed $F'$-finite. The most natural approach would be to
show that for every object $\bX$ of $\mcA\otimes_FF'$ there exist objects $X_0$ and $X^0$
of $\mcA$ together with an epimorphism $t(X_0)\to\bX$ and a monomorphism
$\bX\to t(X^0)$, since then by $F'/F$-full faithfulness of $t$ and $F$-finiteness
of $\mcA$ it would follow that
\[\dim_{F'}\End(\bX)\le\dim_{F'}\Hom\big(t(X_0),t(X^0)\big)=\dim_F\Hom(X_0,X^0)<\infty.\]

\begin{lem}\label{lem:presandcopres1}
For every object $\bX$ of $\mcA\otimes_FF'$
there exists an object $X_0$ of $\mcA$ and an epimorphism
$t(X_0)\to\bX$. 
\end{lem}
\begin{proof}
Let $\bX=(X,\phi)$ be an object of $\mcA\otimes_FF'$. By Lemma \ref{lem:indcatgen}(b)
the object $X$ is the union of its subobjects $X_i$ lying in $\mcA$. Consider the
homomorphisms $h_i:\:t(X_i)\subset t(X)\to\bX$ given as in Remark \ref{rem:tXbXcan}.
Since the image of $h_i$ contains $X_i$ and $\bX$
has finite length, there exists an index $i$ such that $t(X_i)\to\bX$ is an epimorphism,
so we may choose $X_0:=X_i$ together with the epimorphism $h_i$.
%
\end{proof}

However, for a general object $\bX$ of $\mcA\otimes_FF'$ it seems difficult to find a monomorphism $\bX\to t(X^0)$
with $X^0$ an object of $\mcA$. Hence we modify the natural approach sketched above,
by first ``reducing to the simple case'' and then using the fact the endomorphism
algebra of a simple object is Frobenius together with
Proposition \ref{prop:ridealsubob} to find a monomorphism as desired in the simple case.

\begin{dfn}
Let $X$ be an object of a finite abelian category.
\begin{enumerate}
  \item The \emph{socle} $\soc(X)$ of $X$
  is the sum of its simple subobjects.
  Note that $X$ is semisimple if and only if $X=\soc(X)$.
  \item The \emph{socle filtration} of $X$
  is the ascending exhaustive filtration defined as follows: We set $\soc^0(X):=0$,
  $\soc^1(X):=\soc(X)$ and for $i\ge1$ recursively \[\soc^{i+1}(X):=\pi_i^{-1}\big(\soc(X/\soc^iX)\big)\]
  where $\pi_i$ is the canonical projection of $X$ onto $X/\soc^i(X)$.
  \item The \emph{semisimplification} $X^\mathrm{ss}:=\bigoplus_{i\ge 1}\soc^i(X)/\soc^{i-1}(X)$
  of $X$ is the object underlying the graded
  object associated to the socle filtration of $X$.
  \item The  \emph{socle length} of $X$ is the smallest integer $\ell$
  such that $\soc^\ell(X)=X$. We denote it by $\slg(X)$.
\end{enumerate}
The assignments (a-c) are functorial in $X$.
\end{dfn}

We may now ``reduce to the simple case''.

\begin{prop}\label{prop:estimateEnd}
Given two objects $X$, $Y$ of a finite $F$-linear abelian category,
we have $\dim_F\Hom(X,Y)\le\dim_F\Hom(\Xss,\Yss)$.
\end{prop}

\begin{proof}
We proceed by induction on $\ell:=\max\{\slg(X),\slg(Y)\}$.
If $\ell\le 1$ we have $X=\Xss$ and $Y=\Yss$, so the
statement of this proposition is trivial.

Assume that $\ell\ge 2$. For every $f\in\Hom(X,Y)$ we have $f(\soc X)\subset\soc Y$,
so we obtain a diagram
\begin{equation}\label{eqn:estimatinghomss}\vcenter{\xymatrix{
0\ar[r] & \soc X \ar[r] \ar[d]_{f\mid_{\soc X}} & X \ar[r] \ar[d]^f & X/\soc X \ar[r] \ar[d]^{[f]_{X/\soc X}} & 0\\
0\ar[r] & \soc Y \ar[r] & Y \ar[r] & Y/\soc Y\ar[r] & 0
}}\end{equation}
Let $K$ denote the kernel of the induced homomorphism
\[\begin{array}{ccc}
\Hom(X,Y) & \to & \Hom(\soc X,\soc Y)\oplus\Hom(X/\soc X,Y/\soc Y),\\
 f & \mapsto &(f\mid_{\soc X},[f]_{X/\soc X}).
 \end{array}\]
 For $f\in K$ the Snake Lemma applied to (\ref{eqn:estimatinghomss})
 gives us an exact sequence
 \[0\to\soc X\to\ker f\to X/\soc X\arrover{\delta(f)}\soc Y\to\coker f\to Y/\soc Y\to 0.\]
 Since $\delta(f)$ is natural in $f\in K$ we obtain an $F$-linear homomorphism \[\delta:\:K\To\Hom(X/\soc X,\soc Y).\]
 If $\delta(f)=0$, then $\lg(\ker f)=\lg(\soc X)+\lg(X/\soc X)=\lg(X)$, so $\ker f=X$ and $f=0$.
 This shows that $\delta$ is injective.
 
 Now the definition of $K$ and the injectivity of $\delta$ show that
 the dimension of $\Hom(X,Y)$ is bounded above by
 \[\begin{array}{ccc}
\dim_F\Hom(\soc X,\soc Y) & + &\dim_F\Hom(X/\soc X,\soc Y)\\
 & + & \dim_F\Hom(X/\soc X,Y/\soc Y).
\end{array}\]
By construction, all objects involved have socle length $< \ell$, so by 
induction hypothesis the last displayed quantity is bounded above by
\[\begin{array}{ccc}
\dim_F\Hom(\soc X,\soc Y) & + &\dim_F\Hom\big((X/\soc X)^\mathrm{ss},\soc Y\big)\:\:+\\
\dim_F\Hom\big(\soc X,(Y/\soc Y)^\mathrm{ss}\big) & + & \dim_F\Hom\big((X/\soc X)^\mathrm{ss},(Y/\soc Y)^\mathrm{ss}\big),
\end{array}\]
 where we have added an extra term $\dim_F\Hom\left(\soc X,(Y/\soc Y)^\mathrm{ss}\right)\ge 0$.
 However, this last displayed quantity is precisely the dimension of
 \[\Hom(\Xss,\Yss)=\Hom\big(\soc(X)\oplus(X/\soc X)^\mathrm{ss},\soc(Y)\oplus(Y/\soc Y)^\mathrm{ss}\big),\]
 so we have achieved our goal.
\end{proof}

We may now exploit the fact that simple objects have Frobenius endomorphism algebras.

\begin{prop}\label{prop:presandcopres1}
Let $\bX$ be a simple object of $\mcA\otimes_FF'$.
There exists a simple object $S$ of $\mcA$ together
with both an epimorphism $t(S)\to\bX$ and a monomorphism $\bX\to t(S)$.
\end{prop}
\begin{proof}
By Lemma \ref{lem:presandcopres1} there exists
an object $X_0$ of $\mcA$ and an epimorphism $t(X_0)\to\bX$.
Let us first show that we may assume that $X_0$
is simple. If $X_0$ is not simple, then it has length $\ge2$
and we may choose a simple subobject $Y_0\subset X_0$.
Consider the restriction $t(Y_0)\into t(X_0)\to \bX$. If it is
non-zero, then it is an epimorphism because $\bX$ is simple
and we may choose $S:=Y_0$. If it is zero, we obtain a factor homomorphism $t(X_0/Y_0)=t(X_0)/t(Y_0)\to\bX$
which remains an epimorphism. Since $\mcA$ is finite, the claim
follows by induction on the length of $X_0$.

It remains to show that $\bX$ embeds into $t(S)$. Set $E:=\End(S)$.
Since $t$ is $F'/F$-fully faithful by Proposition \ref{lem:tisFpFff}, we may
identify $E':=F'\otimes_FE$ and $\End(t(S))$.
The kernel of our epimorphism $t(S)\to \bX$ corresponds
to a maximal right ideal $I'$ of $E'$ by Proposition \ref{prop:ridealsubob}.
Since $E'/I'$ is simple, it embeds into $E'/\rad(E')$,
the largest semisimple right $E'$-module quotient of $E'$.
Since $E$ is a skew field, it is Frobenius by Proposition
\ref{prop:frobeniusalgebras}(a). Therefore $E'$ is
also Frobenius by Proposition \ref{prop:frobeniusalgebras}(b),
and so Proposition \ref{prop:frobeniusalgebras}(c) shows
that $E'/\rad(E')\isom\soc(E')$, the largest semisimple right $E'$-submodule
of $E'$. Taken together,
we obtain an injection \[E'/I'\into E'/\rad(E')\isom\soc(E')\subset E'.\]
Applying the exact functor $(-)\otimes_{E'}t(S)$, we obtain
an induced monomorphism
\[\bX=\big(E'\otimes_{E'}t(S)\big)/\big(I'\otimes_{E'}t(S)\big)=(E'/I')\otimes_{E'}t(S)\into E'\otimes_{E'}t(S)=t(S),\]
as desired.
\end{proof}

\begin{thm}\label{thm:scalarextwoutup} \begin{enumerate}
  \item $\mcA\otimes_FF'$ is an $F'$-finite $F'$-linear abelian category.
  \item $t:\:\mcA\to\mcA\otimes_FF'$ is an $F$-linear exact $F'/F$-fully faithful functor.
\end{enumerate} \end{thm}
\begin{proof}
(b): We have seen that $t$ is an $F$-linear exact functor. Proposition
\ref{lem:tisFpFff} shows that it is $F'/F$-fully faithful.

(a): By definition, $\mcA\otimes_FF'$ is a finite $F'$-linear abelian category.
It remains to show that $\End(\bX)$ is finite $F'$-dimensional for every object $\bX$ of $\mcA\otimes_FF'$.
Since $\dim_{F'}\End(\bX)\le\dim_{F'}\End\big((\bX)^\ss\big)$ by Proposition \ref{prop:estimateEnd},
it is sufficient to show that $\End(\bX)$
is finite-dimensional for all simple $\bX$. By Proposition \ref{prop:presandcopres1} we may choose an
object $S$ of $\mcA$, an epimorphism $t(S)\to\bX$ and a monomorphism $\bX\to t(S)$. We obtain an $F'$-linear
injection
\[\End(\bX)=\Hom(\bX,\bX)\into\Hom\big(t(S),t(S)\big)\stackrel{(b)}{\isom} F'\otimes_F\End_{\mcA}(S).\]
The target is finite-dimensional since $\mcA$ is $F$-finite, thus so is the source.
\end{proof}

\subsection{Universal Property for Abelian Categories}
Let $F'/F$ be a field extension, and consider an $F$-finite $F$-linear abelian category $\mcA$.
By Theorem \ref{thm:scalarextwoutup} we have an associated $F'$-finite $F'$-linear abelian category $\mcA\otimes_FF'$
and an $F$-linear exact functor $t:\:\mcA\to\mcA\otimes_FF'$.

The goal of this subsection is to show that this functor is ``universal'' among right exact $F'$-linear functors
with target an $F'$-linear abelian category. By this we mean that every such functor $V:\:\mcA\to\mcB$
``factors'' through $\mcA\otimes_FF'$ via a right exact $F'$-linear functor
$V:\:\mcA\otimes_FF'\to\mcB$, and does so ``uniquely'':
\begin{equation*}
\xymatrix{
\mcA\ar[dr]_V \ar[rr]^{t\phantom{m}} && \mcA\otimes_FF'\ar@{.>}[dl]^{V'}\\
 & \mcB}
\end{equation*}
Since we are working with functors,
we have to be more precise in stating this universal property.

\begin{thm}\label{thm:univpropscalex}
Let $\mcB$ be an $F'$-linear abelian category, and consider a right exact $F$-linear
functor $V:\:\mcA\to\mcB$. Then:
\begin{enumerate}
  \item There exists a right exact $F'$-linear functor $V':\:\mcA\otimes_FF\to\mcB$ and an isomorphism
  of functors $\alpha:\:V\Rightarrow V'\circ t$.
  \item If $(V_1',\alpha_1)$ and $(V_2',\alpha_2)$ both have the property stated in (a), then there exists
  a unique isomorphism of functors $\beta':\:V_1'\Rightarrow V_2'$ such that $\alpha_{2,X}=\beta'_{t(X)}\circ\alpha_{1,X}$ for every $X\in\mcA$.
\end{enumerate}
\end{thm}

\begin{rem}
One might expect (since $t$ is exact) that if the functor $V$ in the statement of
Theorem \ref{thm:univpropscalex} is exact, then $V'$ is also exact. This
is false in general.
\end{rem}

The idea behind the proof of Theorem \ref{thm:univpropscalex} is to use the purported right exactness of
$V'$ for the proof of its existence. After all, by
Lemma \ref{lem:presandcopres1} and Theorem \ref{thm:scalarextwoutup}, every
object $\bX$ of $\mcA\otimes_FF'$ possesses a presentation
\[t(X_1)\arrover{\sum_i'\lambda_i\otimes f_i}t(X_0)\to\bX\to 0,\]
with $X_0,X_1\in\mcA$, and finitely many $\lambda_i\in F'$ and $f_i\in\Hom_{\mcA}(X_1,X_0)$. Therefore, by right exactness
and $F'$-linearity of $V'$, we should have
\[V'(\bX)\isom\coker\left(V(X_1)\arrover{\sum_i'\lambda_iV(f_i)}V(X_0)\right).\]
However, since there is no canonical such presentation, it seems difficult to verify that this
idea gives us a well-defined functor $V'$ directly. Hence, we take a detour through the respective ind-categories,
where canonical presentations exist.
We begin by supplementing Lemma \ref{lem:indcatgen}.

\begin{dfn}
Let $\mcB$ be an $F$-linear abelian category, and let\[V:\:\mcA\to\mcB\]be an $F$-linear
functor. The \emph{ind-extension} of $V$ is the $F$-linear functor $\ind V:\:\ind\mcA\to\ind\mcB$
mapping an object $(X_i)_{i\in I}$ of $\ind\mcA$ to $\ind V((X_i)_{i\in I}):=(VX_i)_{i\in I}$ in $\ind\mcB$, and a
homomorphism $f=\varprojlim_i\varinjlim_jf_{ij}$ in \[\Hom_{\ind\mcA}((X_i)_{i\in I},(Y_j)_{j\in J})=\varprojlim_i\varinjlim_j\Hom_{\mcA}(X_i,Y_j)\]
to $\ind V(f):=\varprojlim_i\varinjlim_jV(f_{ij})$.
\end{dfn}

\begin{lem}\label{lem:indVre}\begin{enumerate}
  \item $\ind(V)$ is a functor extending $V$ and functorial in $V$.
  \item If $V$ is right exact, then $\ind(V)$ is right exact.
\end{enumerate}
\end{lem}
\begin{proof}
(a): \cite[Proposition 6.1.9]{Kas06}, (b): \cite[Corollary 8.6.8]{Kas06}.
\end{proof}

\begin{lem}\label{lem:canopresen}
Every $F'$-module $\bX$ in $\ind\mcA$ has a functorial
presentation
\[\Pi(\bX):t(X_1)\arrover{d_1}t(X_0)\arrover{d_0}\bX\to 0\]
with $X_0,X_1\in\ind\mcA$.
\end{lem}
\begin{proof}
For every object $\bX=(X,\phi)$ of $(\ind\mcA)_{F'}$, let $\widetilde{\phi}$ denote the
natural surjective homomorphism $t(X)\to\bX$
as in Remark \ref{rem:tXbXcan}.

We may now define our presentation: Given $\bX$ as above, we set $X_0:=X$, and $d_0:=\widetilde{\phi}$.
Then $\ker(d_0)$ is an $F'$-module $(X_1,\phi_1)$, and we set $d_1:=\widetilde{\phi_1}$. We obtain an exact sequence
\[t(X_1)\arrover{d_1}t(X_0)\arrover{d_0}\bX\to0\]
of $F'$-modules, which we denote as $\Pi(\bX)$.

Let us show that $\Pi(\bX)$ is functorial in $\bX$: Given another $F'$-module $\bY=(Y,\psi)$ and a homomorphism $f:\:\bX\to\bY$,
we set $f_0:=\id\otimes f$ and $f_1:=\id\otimes(f_0|_{(X_1,\phi_1)})$. We obtain a diagram
\[\xymatrix{
\Pi(\bX): & F'\otimes_FX_1 \ar[d]^{f_1} \ar[r] & F'\otimes_FX_0 \ar[d]^{f_0}\ar[r] & \bX \ar[d]^f \ar[r] & 0\\
\Pi(\bY): & F'\otimes_FY_1              \ar[r] & F'\otimes_FY_0             \ar[r] & \bY          \ar[r] & 0
}\]
This diagram commutes by definition, so we have constructed a canonical homomorphism $\Pi(f):\:\Pi(\bX)\to\Pi(\bY)$.
\end{proof}

\begin{lem}\label{lem:indVgivesindVtwid}
Let $\ind V:\:\ind\mcA\to\ind\mcB$ be a right exact $F$-linear functor.
Let $t(\ind\mcA)$ denote the full subcategory of $(\ind\mcA)_{F'}$ with
the image of $\ind\mcA$ under $t$ as objects.
There exists a natural $F'$-linear functor
\[\ind\widetilde{V}:\:t(\ind\mcA)\to\ind\mcB\]%
such that $\ind V=\ind\widetilde{V}\circ t$.
\end{lem}
\begin{proof}
Since $\ind\widetilde{V}$ is to extend $\ind V$,
on objects $t(X)$ of $t(\ind\mcA)$ we must and may set
\[\ind\widetilde{V}(t(X)):=\ind V(X).\]
Given two objects $t(X)$ and $t(Y)$ of $\ind\mcA$,
we have $X=\varinjlim_{i\in I}X_i$ and $Y=\varinjlim_{j\in J} Y_j$
for directed sets $I,J$ and objects $X_i,Y_j$ of $\mcA$.
Recall that both $t$ and $\ind V$ are right exact and note
that both commute with direct sums; it follows that both commute
with direct limits -- they are ``continuous''.

Considering first the special case that $X=X_i$ and $Y=Y_j$
for some $i$ and $j$, by Proposition \ref{lem:tisFpFff}
we have $\Hom(tX_i,tY_j)=F'\otimes_F\Hom(X_i,Y_j)$, so since
$\ind\widetilde{V}$ is to be $F'$-linear and extend $\ind V$
we see that $\ind\widetilde{V}:\:\Hom(t(X_i),t(Y_j))\to\Hom(\ind V(X_i),\ind V(Y_j))$
must be the $F'$-linear extension of $\ind V:\:\Hom(X_i,Y_j)\to\Hom(\ind V(X_i),\ind V(Y_j))$.

Returning to the general case, using the special case
given above for all $i$ and $j$ we define $\ind\widetilde{V}$
on homomorphisms as
\begin{eqnarray*}
\Hom(t(X),t(Y)) & = & \Hom (\varinjlim_i t(X_i),\varinjlim_j t(Y_j)),\quad\text{since $t$ is continuous}\\
                & = & \varprojlim_i\varinjlim_j\Hom(t(X_i),t(Y_j)),\quad\text{by definition}\\
                &\to& \varprojlim_i\varinjlim_j\Hom(\ind V(X_i),\ind V(Y_j)),\quad\text{special cases}\\
                & = & \Hom(\varinjlim_i\ind V(X_i),\varinjlim_j\ind V(Y_j)),\quad\text{by definition}\\
                & = & \Hom(\ind V(X),\ind V(Y)),\quad\text{$\ind\widetilde{V}$ is continuous}
\end{eqnarray*}
The conscientious reader will check that our definition
of $\ind\widetilde{V}$ is well-defined. Obviously, it extends
$\ind V$ in the sense that $\ind\widetilde{V}\circ t=\ind V$.
\end{proof}

\begin{lem}\label{lem:VgivesindVp}
Let $\ind V:\:\ind\mcA\to\ind\mcB$ be a right exact $F$-linear functor.
There exists a right exact $F'$-linear functor
\[\ind V':\:(\ind\mcA)_{F'}\to\ind\mcB\]%
and an isomorphism of functors
$\ind\alpha: \ind V\Rightarrow(\ind V')\circ t$.
\end{lem}
\begin{proof}
%
By Lemma \ref{lem:indVgivesindVtwid}, our given
functor $\ind V$ extends naturally to an
$F'$-linear functor \[\ind\widetilde{V}:\: t(\ind\mcA)\to\ind\mcB,\]
where $t(\ind\mcA)$ denotes the full subcategory of $(\ind V)_{F'}$ which
has as objects the essential image of $\ind\mcA$ under $t$.

In particular, given an $F'$-module $\bX=(X,\phi)$ in $\ind\mcA$, we may apply $\ind\widetilde{V}$ to the portion $t(X_1)\arrover{d_1}t(X_0)$
of the presentation $\Pi(\bX)$ given by Lemma \ref{lem:canopresen}, and set
\[\ind V'(\bX):=\coker\left(\ind V(X_1)\arrover{\ind\widetilde{V}(d_1)}\ind V(X_0)\right).\]
Given a second object $\bY$ and a homomorphism $f:\:\bX\to\bY$ in of $F'$-modules, we may apply
$\ind\widetilde{V}$ to the portion
\[\xymatrix{
t(X_1) \ar[d]^{f_1} \ar[r] & t(X_0) \ar[d]^{f_0}\\
t(Y_1)              \ar[r] & t(Y_0)             
}\]
of the homomorphism $\Pi(f)$ of presentations given in the proof of Lemma \ref{lem:canopresen}. Now the
universal property of cokernels implies that there is exactly one homomorphism $\ind V'(f):\:\ind V'(\bX)\to \ind V'(\bY)$
completing the image of the above commutative square under $\ind\widetilde{V}$ to a commutative diagram:
\[\xymatrix{
\ind\widetilde{V}(t(X_1)) \ar[d]^{\ind\widetilde{V}(f_1)} \ar[r] & \ind\widetilde{V}(t(X_0)) \ar[d]^{\ind\widetilde{V}(f_0)} \ar[r] & \ind V'(\bX) \ar@{.>}[d]^{\ind V'(f)} \ar[r] & 0\\
\ind\widetilde{V}(t(Y_1))                             \ar[r] & \ind\widetilde{V}(t(Y_0))              \ar[r] & \ind V'(\bY) \ar[r] & 0
}\]
The universal property of cokernels also shows that $\ind V'(\id_{\bX})=\id_{\ind V'(\bX)}$ for all $\bX$, and that $\ind V'(gf)=\ind V'(g)\ind V'(f)$ for all pairs of composable homomorphisms $\bX\arrover{f}\bY\arrover{g}\bZ$, so $\ind V'$ is indeed a functor $(\ind\mcA)_{F'}\to\ind\mcB$.

Let us prove that $\ind V'$ is right exact, so let $\bX\arrover{f}\bY\arrover{g}\bZ\to 0$
be a right exact sequence in $(\ind\mcA)_{F'}$. We obtain the following commutative diagram:
\[\def\objectstyle{\scriptstyle}
\def\labelstyle{\scriptstyle}
\xymatrix{
\ind V(X_1) \ar[r] \ar[d] & \ind V(X_0) \ar[r] \ar[d] & \ind V'(\bX) \ar[r] \ar[d]^{\ind V'(f)} & 0\\
\ind V(Y_1) \ar[r] \ar[d] & \ind V(Y_0) \ar[r] \ar[d] & \ind V'(\bY) \ar[r] \ar[d]^{\ind V'(g)} & 0\\
\ind V(Z_1) \ar[r] \ar[d] & \ind V(Z_0) \ar[r] \ar[d] & \ind V'(\bZ) \ar[r] \ar[d]         & 0\\
0 & 0 & 0
}\]
The rows are the sequences defining $\ind V'$ on objects, so they are exact by definition.
Since $V$ is right exact and the vertical homomorphisms arise from $\ind\mcA$, the first two columns are exact.
Hence, by the $3\times 3$-Lemma, the remaining column is exact, which is what we had to prove.

Finally, let us construct an isomorphism $\ind\alpha: \ind V\Rightarrow(\ind V')\circ t$ of functors.
We let $K$ be the kernel of the multiplication
$\mu$ of $F'$, so we have an exact sequence of $F'$-vector spaces
\[K\to F'\otimes_FF'\arrover{\mu}F'\to 0.\]
For every object $X$ of $\ind\mcA$, this induces an exact sequence
\begin{eqnarray}\label{eqn:muotimesVX}\def\objectstyle{\scriptstyle}
\def\labelstyle{\scriptstyle}
\xymatrix@C=9pt{
&&K\otimes_{F'}\ind V(X)\ar[r]& (F'\otimes_FF')\otimes_{F'}\ind V(X)\ar[r]& F'\otimes_{F'}\ind V(X)\ar[r]& 0
}
\end{eqnarray}
in $\ind\mcB$. We use this observation to construct the following diagram:
\[
\def\objectstyle{\scriptstyle}
\def\labelstyle{\scriptstyle}\xymatrix{
\ind \widetilde{V}(t(X_1)) \ar@{=}[d] \ar[r]^{\ind \widetilde{V}(d_1)} & \ind \widetilde{V}(t(X_0)) \ar@{=}[d] \ar[r] & \ind V'(t(X)) \ar[r]   & 0\\
\ind V(K\otimes_FX) \ar[r] \ar[d]^{\isom}                                 & \ind V(F'\otimes_FX) \ar[d]^{\isom} \\
K\otimes_F\ind V(X) \ar[r] \ar@{->>}[d]                                   & F'\otimes_F\ind V(X) \ar[d]^{\isom} \\
K\otimes_{F'}\ind V(X) \ar[r]                                             & (F'\otimes_FF')\otimes_{F'}\ind V(X) \ar[r]          & F'\otimes_{F'}\ind V(X) \ar[r] & 0\\
}\]
The first row is the definition of $\ind V'(t(X))$, which we unravel in the second row.
The isomorphisms connecting the second and third row are canonical, as are the epimorphism and the isomorphism
connecting the third row with the fourth, which is the exact sequence (\ref{eqn:muotimesVX}). One can check
that this diagram commutes, so by the Five Lemma we obtain a canonical isomorphism $\ind V'(t(X))\to F'\otimes_{F'}\ind V(X)$.
Precomposing the inverse of this isomorphism with the canonical isomorphism $\ind V(X)\to F'\otimes_{F'}\ind V(X)$, we obtain an isomorphism
\[(\ind\alpha)_{X}:\quad \ind V(X)\arrover{\isom}\ind V'(t(X)),\]
as desired. By construction, $(\ind\alpha)_X$ is natural in $X$, so $\ind\alpha$ is a homomorphism of functors.
Therefore $\ind\alpha$ is an isomorphism of functors, since we have already seen that $(\ind\alpha)_X$ is an isomorphism for each $X\in\mcA$.
\end{proof}

\begin{lem}\label{lem:indVpgivesVp}
Let $\mcA$ be an $F$-finite $F$-linear abelian category and $V:\:\mcA\to\mcB$ be a right exact $F$-linear
functor. Let $\ind V'$ be the right exact $F'$-linear functor associated to $\ind V$
via Lemma \ref{lem:VgivesindVp}. Then there exists a functor \[V':\:\mcA\otimes_FF'\To\mcB\]
such that $V'$ fulfills the requirements of Theorem \ref{thm:univpropscalex}(a) and
the following diagram commutes up to isomorphism of functors:
\[\xymatrix{
(\ind\mcA)_{F'} \ar[rr]^{\ind V'} && \ind\mcB\\
\mcA\otimes_FF' \ar@{}[u]|{\bigcup} \ar@{.>}[rr]^{V'} && \mcB \ar[u]
}\]
\end{lem}
\begin{proof}
By Lemma \ref{lem:indVre}, $V$ induces a right exact $F$-linear functor $\ind V:\:\ind\mcA\to\ind\mcB$.
By Lemma \ref{lem:VgivesindVp}, $\ind V$ induces a right exact $F'$-linear functor
$\ind V':\:(\ind\mcA)_{F'}\to\ind\mcB$. We obtain the following diagram, which commutes
up to isomorphism of functors:
\[\xymatrix{
\mcA \ar[r]\ar[d]_V & \ind\mcA \ar[r]\ar[d]_{\ind V} & (\ind\mcA)_{F'}\ar[dl]^{\ind V'}\\
\mcB \ar[r]         & \ind\mcB
}\]
We let $V'$ be the restriction of $\ind V'$ to $\mcA\otimes_FF'\subset(\ind\mcA)_{F'}$. If
we prove that the image of $V'$ lies in the essential image of $\mcB$ in $\ind\mcB$, then we will
have shown that the following diagram commutes up to isomorphism of functors:
\[\xymatrix{
\mcA \ar[dr]_{V} \ar[r] & \mcA\otimes_FF' \ar@{.>}[d]^{V'} \ar[r] & (\ind\mcA)_{F'} \ar[d]^{\ind V'} \\
                        & \mcB \ar[r]                             & \ind\mcB
}\]
So let us do this: Given $\bX$ in $\mcA\otimes_FF'$, by Lemma \ref{lem:presandcopres1}
and Theorem \ref{thm:scalarextwoutup} there exists a right exact sequence
\[t(X_1)\arrover{g}t(X_1)\to \bX\to 0\]
in $(\ind\mcA)_{F'}$, with $X_0,X_1\in\mcA$ and $g\in F'\otimes\Hom_\mcA(X_0,Y_0)$.
Since $\ind V'$ is right exact, and its restriction to $\mcA$ is isomorphic to $V$, the induced sequence
\[V(X_1)\arrover{(\ind V')(g)} V(X_0)\to V'(\bX)\to 0\]
is exact in $\ind\mcB$. Since $\mcB\to\ind\mcB$ is exact, it follows that $V'(\bX)$
is isomorphic to the cokernel of the homomorphism $\ind V'(g)$ as calculated in
the full subcategory $\mcB$.
\end{proof}

We turn to the unicity of our extensions $\ind V'$ and $V'$.

\begin{lem}\label{lem:unicofindvip}
Let $\ind V_1,\,\ind V_2:\:\ind\mcA\to\ind\mcB$ be two right exact
$F$-linear functors, and let $(\ind V_1',\ind\alpha_1)$, $(\ind V_2',\ind\alpha_2)$
be extensions as in Lemma \ref{lem:VgivesindVp} of $\ind V_1$, $\ind V_2$,
respectively.

For every homomorphism of functors $\ind\beta:\:\ind V_1\Rightarrow \ind V_2$
there exists a unique homomorphism of functors $\ind \beta':\: \ind V_1'\Rightarrow\ind V_2'$
such that $\ind\alpha_{2,X}\circ \ind\beta_{X}=\ind\beta'_{t(X)}\circ \ind\alpha_{1,X}$ for all $X\in\ind\mcA$.

Moreover, $\ind\beta$ is a monomorphism (resp. epimorphism, resp. isomorphism) if and only if
$\ind\beta'$ is.
\end{lem}
\begin{proof}
For $\bX\in(\ind\mcA)_{F'}$, the sequences $\ind V_i'(\Pi(\bX))$ are both exact, since both $\ind V_i'$ are
right exact by assumption. They are connected by means of the following commutative diagram
with exact rows:
\[\xymatrix{
\ind V_1'(t(X_1))\ar[r]\ar[d]^{(\ind \alpha_{1,X_1})^{-1}} & \ind V_1'(t(X_0)) \ar[r]\ar[d]^{(\ind \alpha_{1,X_0})^{-1}} & \ind V_1'(\bX) \ar[r] & 0\\
\ind V(X_1) \ar[r]\ar[d]^{\ind \beta_{X_1}} & \ind V(X_0)\ar[d]^{\ind \beta_{X_0}}\\
\ind V(X_1) \ar[r]\ar[d]^{\ind \alpha_{2,X_1}} & \ind V(X_0)\ar[d]^{\ind \alpha_{2,X_0}}\\
\ind V_2'(t(X_1))\ar[r]                                      & \ind V_2'(t(X_0)) \ar[r]                           & \ind V_2'(\bX) \ar[r] & 0\\
}\]
By the universal property of cokernels, we obtain a unique homomorphism $\ind \beta'_{\bX}:\:V_1'(\bX)\to V_2'(\bX)$ completing the diagram
to a homomorphism of right exact sequences. By the Five Lemma, $\ind\beta'_{\bX}$ is a monomorphism
(resp. epimorphism, resp. isomorphism) if and only if $\ind\beta$ is.
Now by construction $\ind \beta'_{\bX}$ is natural in $\bX$, so $\ind\beta':\:\ind V_1'\Rightarrow \ind V_2'$
is a homomorphism of functors, which is a monomorphism
(resp. epimorphism, resp. isomorphism) if and only if $\ind\beta$ is.

The same diagram shows that any homomorphism $\ind V_1'\Rightarrow\ind  V_2'$ which restricts
to $(\ind \alpha_2)\circ(\ind\beta)\circ(\ind \alpha_1)^{-1}$ on $t(\ind \mcA)$ must coincide with $\ind \beta'$.

It remains to show that $\ind \beta'$ restricts in such a way. But this again follows from the same diagram, since if $\bX=t(\widetilde{X})$
for $\widetilde{X}\in\ind \mcA$, then $\ind \alpha_{2,\widetilde{X}}\circ\ind\beta_{X}\circ(\ind \alpha_{1,\widetilde{X}})^{-1}$
fits in the same place as $\ind \beta'_{t(\widetilde{X})}$, so the two homomorphisms must coincide
by the universal property of cokernels.
\end{proof}

\begin{lem}\label{lem:unicofvip}
Given two pairs $(V_i',\alpha_i)$, $(V_2',\alpha_2)$ extending $V$
as in Theorem \ref{thm:univpropscalex}(a), there exists a unique isomorphism of functors $\beta':\: V_1'\Rightarrow V_2'$
such that $\beta'_{t(X)}\circ\alpha_{1,X}=\alpha_{2,X}$ for all $X\in\mcA$.
\end{lem}
\begin{proof}
Given two such pairs of data $(V_i',\alpha_i)$, using Lemma \ref{lem:indVre} we obtain two pairs of data
$(\ind V_i',\ind\alpha_i)$ extending $\ind V$ as in Lemma \ref{lem:VgivesindVp}.
Lemma \ref{lem:unicofindvip} applied to $\ind\beta:=\id_{\ind V}$ shows that there exists an isomorphism of functors
$\ind\beta':\:\ind V_1'\Rightarrow\ind V_2'$ such that
$\ind \beta'_{t(X)}\circ\ind\alpha_{1,X}=\ind \alpha_{2,X}$ for all $X\in\ind\mcA$.
The restriction $\beta'$ of $\ind\beta'$ to $\mcA\otimes_FF'\subset(\ind\mcA)_{F'}$ is then
an isomorphism of functors $V_1'\Rightarrow V_2'$ with the required properties.

Let us show that this $\beta'$ is unique. Given two isomorphisms of functors $\beta'_1,\beta'_2: V_1'\Rightarrow V_2'$
with an identification of isomorphisms
\[\beta'_1\mid_{t(\mcA)}=\alpha_2\circ\alpha_1^{-1}=\beta'_2\mid_{t(\mcA)}:\:V_1'\Rightarrow V_2',\]
applying $\ind(-)$ gives us an identification of isomorphisms
\[\ind\beta'_1\mid_{t(\ind\mcA)}=\ind(\alpha_2\circ\alpha_1^{-1})=\ind\beta'_2\mid_{t(\ind\mcA)}:\:\ind V_1'\Rightarrow \ind V_2'\]
by Lemma \ref{lem:indVre}, where $\ind(\alpha_2\circ\alpha_1^{-1})=\ind\alpha_2\circ\ind\alpha_1^{-1}$ and clearly $t(\ind\mcA)=\ind t(\mcA)$.
Lemma \ref{lem:unicofindvip} shows that $\ind\beta'_1=\ind\beta'_2$,
so restricting to $t(\mcA)$ we obtain $\beta'_1=\beta'_2$ as desired.
\end{proof}

\begin{proof}[Proof of Theorem \ref{thm:univpropscalex}]
(a): Lemma \ref{lem:indVpgivesVp}, (b): Lemma \ref{lem:unicofvip}.
\end{proof}

\begin{prop}\label{propunivpropscalexfuncinV}
Let $\mcA$ be an $F$-finite $F$-linear abelian category, $\mcB$ an $F'$-linear abelian category, and
$V_1',V_2':\:\mcA\otimes_FF'\to\mcB$ two right exact $F'$-linear functors. Then for every
homomorphism of functors $\beta:\:V_1'\circ t\Rightarrow V_2'\circ t$ there exists
a unique homomorphism of functors $\beta':\:V_1'\Rightarrow V_2'$ extending $\beta$.

Moreover, $\beta$ is a monomorphism (resp. epimorphism, resp. isomorphism) if
and only if $\beta'$ is such.
\end{prop}
\begin{proof}
This may be deduced from Lemma \ref{lem:unicofindvip}
as in the proof of Lemma \ref{lem:unicofvip}.
\end{proof}

\subsection{Permanence of Semisimplicity on Objects}\label{ss:permofss}

Let $F$ be  a field. Recall that an exact functor
between abelian categories is semisimple on objects if
it maps semisimple objects to semisimple objects. Recall
also the notion of separability as given in Definition \ref{dfn:sepalgebra}.


\begin{prop}\label{prop:FpFsepthentsemsimple}
Let $\mcA$ be an $F$-finite $F$-linear abelian category. Assume that $F'/F$
is a separable field extension. Then $t:\:\mcA\to\mcA\otimes_FF'$
is semisimple on objects.
\end{prop}
\begin{proof}
Let $X$ be a semisimple object of $\mcA$, and set $E:=\End_{\mcA}(X)$.
It is a semisimple finite-dimensional $F$-algebra by assumption. We must show that $t(X)$
is semisimple, and may assume that $X$ is simple since $t$ is additive.
By Proposition \ref{prop:ridealsubob}, $t(X)$ is semisimple if and only
if the $F'$-algebra $F'\otimes_FE$ is semisimple. And $F'\otimes_FE$
is semisimple by Proposition \ref{prop:separablealgebras}(a)
since $E$ is semisimple and $F'/F$ is separable .
\end{proof}

Even if $F'/F$ is not separable, $t(X)$ may be semisimple:


\begin{prop}\label{prop:semisimpllll}
Let $\mcA$ be an $F$-finite $F$-linear abelian category. Let $X$ be
a semisimple object of $\mcA$ for which $\End(X)$ is a separable
$F$-algebra. Then $t(X)$ is semisimple.
\end{prop}
\begin{proof}
The algebra $F'\otimes_FE$ is semisimple by Proposition \ref{prop:separablealgebras}(b)
since $E$ is separable and semisimple. So
we may repeat the proof of Proposition \ref{prop:FpFsepthentsemsimple},
mutatis mutandis.
\end{proof}

\begin{prop}\label{prop:permofsemisimpli}
Let $\mcA$ be an $F$-finite $F$-linear abelian category, $\mcB$ an $F'$-linear abelian
category, $V:\:\mcA\to\mcB$ a right exact $F$-linear functor and
$V':\:\mcA\otimes_FF'\to\mcB$ the induced right exact $F'$-linear functor.
Assume that $F'/F$ is a separable field extension.
Then $V$ is semisimple on objects if and only if $V'$ is semisimple on objects.
\end{prop}
\begin{proof}
If $V'$ is semisimple on objects, then so is $V=V'\circ t$, a composition
of such functors by Proposition \ref{prop:FpFsepthentsemsimple}.
Conversely, if $V$ is semisimple on objects, let $\bX$ be a semisimple
object of $\mcA\otimes_FF'$. We must show that $V'(\bX)$ is semisimple,
and may assume that $\bX$ is simple, since $V'$ is additive.
By Proposition \ref{prop:presandcopres1}(a) there exists a simple object
$S$ of $\mcA$ such that $\bX$ is a quotient of $t(S)$. Since $V'$
is right exact, this implies that $V'(\bX)$ is a quotient of $V'(t(S))=V(S)$,
which is semisimple by assumption. Therefore, $V'(\bX)$ itself is semisimple.
\end{proof}

\section{Tensor Categories}

\subsection{Scalar Extension of Tensor Categories}\label{ss:scalextoftenscats}
Let $F'/F$ be a field extension, and consider an $F$-finite $F$-linear abelian category $\mcA$
with associated scalar extension functor $t:\:\mcA\to\mcA\otimes_FF'$.

\begin{dfn}
An \emph{$F$-multilinear endofunctor of $\mcA$} is a functor
\[M:\:\mcA^{\times n}\To\mcA\]
which is $F$-linear in each each argument, for some integer $n\ge 1$.
\end{dfn}

\begin{prop}\label{prop:univpropscalexbilin}
Let $\mcA$ be an $F$-finite $F$-linear abelian category, and $n\ge 1$ an integer.
\begin{enumerate}
  \item Let $M:\:\mcA^{\times n}\to\mcA$ be a right exact $F$-multilinear functor. Then there exists a
  			right exact $F'$-multilinear functor $M':\:(\mcA\otimes_FF')^{\times n}\to\mcA\otimes_FF'$
  			together with an isomorphism $\alpha:\:t\circ M\Rightarrow M'\circ (t^{\times n})$ of functors.
  \item Let $M_1,M_2:\:\mcA^{\times n}\to\mcA$ be two right exact $F$-multilinear functors, and let $(M_1',\alpha_1)$, $(M_2',\alpha_2)$
  			be extensions as in (a) of $M_1$, $M_2$ respectively. Then, for every homomorphism of functors $\beta:\:M_1\Rightarrow M_2$
  			there exists a unique homomorphism of functors $\beta':\:M_1'\Rightarrow M_2'$ such that
  			$t\beta\circ\alpha_1=\alpha_2\circ \beta'_{t^{\times n}}$ in the sense that for every $n$-tuple of objects
  			$(X_1,\ldots X_n)\in\mcA^{\times n}$ the following diagram commutes:
  			\[\xymatrix{
  			M_1'\big(t(X_1),\ldots,t(X_n)) \ar[rr]^{\phantom{mm}\alpha_{1,(X_1,\ldots,X_n)}} \ar[d]_{\beta'_{(t X_1,\ldots,t X_n)}}
  			&& t(M_1(X_1,\ldots,X_n))\ar[d]^{\id\otimes\beta_{(X_1,\ldots,X_n)}}\\
  			M_2'(t(X_1),\ldots,t(X_n)) \ar[rr]^{\phantom{mm}\alpha_{2,(X_1,\ldots,X_n)}}
  			&& t(M_1(X_1,\ldots,X_n))
  			}\]
\end{enumerate}
\end{prop}
\begin{proof}
This is one of the proofs in mathematics which does not become much clearer by writing it
down in detail. The case $n=1$ follows from Theorem \ref{thm:univpropscalex}
and Proposition \ref{propunivpropscalexfuncinV} applied to $V:=t\circ M$.
We settle for a sketch of the construction of $M'$ in the case $n=2$.
We set $\otimes:=M$ and will denote the desired extension $M'$ by $\otimes'$.
Let us abbreviate notation by setting $\mcA':=\mcA\otimes_FF'$.

For every $Y\in\mcA$, let
\[-\otimes't(Y):=(t\circ(-\otimes Y))':\:\mcA'\to\mcA'\]
denote the scalar extension of $t\circ(-\otimes Y)$ as in Theorem \ref{thm:univpropscalex}(a).
It is an $F'$-linear right exact functor. It is also functorial in $Y$, since a homomorphism
$f:\:Y_1\to Y_2$ induces a homomorphism of functors $t\circ(-\otimes Y_1)\Rightarrow t\circ(-\otimes Y_2)$
given for $X\in\mcA$ by $\id\otimes f:\:X\otimes Y_1\to X\otimes Y_2$. By Proposition
\ref{propunivpropscalexfuncinV} this induces a unique homomorphism of functors $-\otimes't(Y_1)\Rightarrow-\otimes't(Y_2)$.
Therefore, we obtain a right exact functor
\[-\otimes't(-):\:\mcA'\times\mcA\to\mcA',\quad (\bX,Y)\mapsto \bX\otimes't(Y)\]
which is $F'$-linear in the first variable and $F$-linear in the second.

For every $\bX\in\mcA'$, let
\[\bX\otimes'(-):=(\bX\otimes'-)':\:\mcA'\to\mcA'\]
denote the scalar extension of $\bX\otimes'-$ as in Theorem \ref{thm:univpropscalex}(a).
It is an $F'$-linear right exact functor. By similar reasoning as before, it is
functorial in $\bX$. Therefore, we obtain a right exact $F$-bilinear functor
\[(-)\otimes'(-):\:\mcA'\times\mcA'\to\mcA',\quad (\bX,\bY)\mapsto \bX\otimes'\bY.\]
It fulfills what is required in item (a).
\end{proof}

For an introduction to the theory of tensor categories,
we refer to \cite{DeM82} and \cite{Del90}. We will repeat only the
definitions to fix notation. 

\begin{dfn}
\begin{enumerate}
\item
An \emph{abelian tensor category} is an abelian category
$\mcA$ together with a right exact biadditive functor $\otimes:\:\mcT\times\mcT\to\mcT$,
its \emph{tensor product}, which is assumed to be equipped with sufficiently many (associativity, commutativity and unity) constraints
such that the tensor product of an unordered finite set of objects
is well-defined. In particular, there exists a unit object $\bbu$.
One tends to suppress mention of the constraints.
\item An \emph{abelian tensor category over $F$} is an abelian tensor
category together with a ring isomorphism $F\to\End(\bbu)$.
Using this isomorphism and the constraints, $\mcT$ becomes $F$-linear
and $\otimes$ is $F$-bilinear.
\item A \emph{tensor functor} is a functor $V:\:\mcS\to\mcT$
between two abelian tensor categories, which is assumed to be equipped with tensor constraints,
that is, canonical isomorphisms $V(X)\otimes V(Y)\isom V(X\otimes Y)$
compatible with the constraints of $\mcS$ and $\mcT$.
\item A \emph{morphism of tensor functors} $V,W:\:\mcS\to\mcT$
is a natural transformation $\eta:\:V\Rightarrow W$, which is assumed to be
compatible with the tensor constraints of $V$ and $W$. We let
$\Hom^\otimes(V,W)$ denote the set of morphisms of tensor functors
$V\Rightarrow W$ and let $\Aut^\otimes(V)$ denote the group of automorphisms
of $V$ as tensor functor.
\end{enumerate}
\end{dfn}

Given an $F$-finite abelian tensor categoy $(\mcT,\otimes)$ over $F$,
Proposition \ref{prop:univpropscalexbilin}
provides a natural candidate for a tensor product $\otimes'$
on $\mcT\otimes_FF'$. The following proposition demonstrates
that our instincts are correct.

\begin{thm}\label{thm:scalextoftensorcats}
Let $(\mcT,\otimes)$ be an $F$-finite abelian tensor category over $F$.
Let $F'/F$ be a field extension.
\begin{enumerate}
  \item $(\mcT\otimes_FF',\otimes')$ is an abelian tensor category over $F'$.
  \item $t:\:\mcT\to\mcT\otimes_FF'$ is a tensor functor.
\end{enumerate}
\end{thm}
\begin{proof}
By assumption, $\otimes:\:\mcT\times\mcT\to\mcT$ is a right exact $F$-bilinear
functor and comes equipped with an associativity constraint $\phi$, commutativity
constraint $\psi$, unit object $\bbu$ and isomorphism $F\to\End(\bbu)$.
Set $\mcT':=\mcT\otimes_FF'$.
By Proposition \ref{prop:univpropscalexbilin}, the induced functor $\otimes':\:\mcT'\times\mcT'\to\mcT'$
is right exact and $F'$-bilinear.
The associativity constraint $\phi:\:\otimes\circ(\id\times\otimes)\Rightarrow\otimes\circ(\otimes\times\id)$
has a unique extension to an isomorphism of functors
$\phi':\:\otimes'\circ(\id\times\otimes')\Rightarrow\otimes'\circ(\otimes'\times\id)$
by Proposition \ref{prop:univpropscalexbilin} for $n=3$, and the commutativity
constraint $\psi:\:\otimes\Rightarrow\otimes\circ s$ has a unique extension to an
isomorphism of functors $\psi':\:\otimes'\Rightarrow\otimes'\circ s$ by Proposition \ref{prop:univpropscalexbilin} for $n=2$,
where $s$ denotes the ``switch'' functor $\mcC\times\mcC\to\mcC\times\mcC,\:(X,Y)\mapsto (Y,X)$ for any
category $\mcC$.

It remains to check that three relations
hold among $\phi'$ and $\psi'$ (namely, $\psi'\circ\psi'=\id$, the Pentagon Axiom and the Hexagon Axiom), and
that there exists a unit object $\bbu'\in\mcT'$ for which $F'\to\End_{\mcT'}(\bbu')$ is an isomorphism.

Each of these three relations state
that certain natural transformations (constructed using $\phi'$ and $\psi'$) of
certain functors $\mcT^{\prime\times n}\to\mcT'$ (constructed using $\otimes'$)
are equal. The first states that $\psi'_{\bY,\bX}\circ\psi'_{\bX,\bY}=\id_{\bX\otimes'\bY}$
for all $\bX,\bY\in\mcT'$.
The Pentagon Axiom states that $\phi'\circ\phi'=(\phi'\otimes'\id)\circ\phi'\circ(\id\otimes'\phi')$
in the sense that for every quadruple $(\bX,\bY,\bZ,\bT)$ of objects of $\mcT'$
the following diagram commutes:
\[\def\objectstyle{\scriptstyle}
\def\labelstyle{\scriptstyle}
\xymatrix{
\bX\otimes'(\bY\otimes'(\bZ\otimes'\bT)) \ar[d]\ar[r] & (\bX\otimes'\bY)\otimes'(\bZ\otimes'\bT) \ar[r] & ((\bX\otimes'\bY)\otimes'\bZ)\otimes'\bT\\
\bX\otimes'((\bY\otimes'\bZ)\otimes'\bT)\ar[rr] && (\bX\otimes'(\bY\otimes'\bZ))\otimes'\bT \ar[u]
}\]
The Hexagon Axiom states that $\phi'\circ\psi'\circ\phi'=(\psi'\otimes\id)\circ\phi'\circ(\id\otimes'\psi')$
in the sense that for every triple $(\bX,\bY,\bZ)$ of objects of $\mcT'$
the following diagram commutes:
\[\def\objectstyle{\scriptstyle}
\def\labelstyle{\scriptstyle}
\xymatrix{
\bX\otimes'(\bY\otimes'\bZ) \ar[d]\ar[r] & (\bX\otimes'\bY)\otimes'\bZ \ar[r] & \bZ\otimes'(\bX\otimes'\bY)\ar[d]\\
\bX\otimes'(\bZ\otimes'\bY) \ar[r]       & (\bX\otimes'\bZ)\otimes'\bY \ar[r] & (\bZ\otimes'\bX)\otimes'\bY
}\]
 
In all cases, Proposition \ref{prop:univpropscalexbilin}(b) and the assumption that $\mcT$ is a tensor
category show that the stated relations hold. Let us prove the first relation $\psi'\circ\psi'=\id$
as an example.
Now $\psi'\circ\psi'$ is a homomorphism of functors $\otimes'\to\otimes'$.
Its restriction to $\otimes$ is equal to $\psi\circ\psi$ by definition, and
is equal to the identity endomorphism of $\otimes$, since $\psi'$ extends $\psi$ and $\mcT$ is a tensor category.
So $\psi'\circ\psi'$ is an extension of the identity endomorphism of $\otimes$. Since the identity endomorphism of $\otimes'$ is
another extension of the identity endomorphism of $\otimes$, Proposition \ref{prop:univpropscalexbilin}(b)
shows that $\psi'\circ\psi'$ and the identity endomorphism of $\otimes'$ coincide!
The proofs that the Pentagon and Hexagon axioms hold are similar, if somewhat more involved notationally.

It remains to show that there exists a unit object of $(\mcT',\otimes')$ with endomorphism ring $F'$,
and we claim that $t(\bbu)$ is one for every unit object $\bbu$ of $(\mcT,\otimes)$. To say that $\bbu$ is a unit object
means that there exists an isomorphism $u:\bbu\to\bbu\otimes\bbu$ and that $\bbu\otimes(-)$ is an equivalence of categories
$\mcT\to\mcT$.

Now $t(u):\:t(\bbu)\to t(\bbu\otimes\bbu)\isom t(\bbu)\otimes't(\bbu)$ is an isomorphism since $t$
is a functor.  Let $V$ be a quasi-inverse of the restriction $\bbu\otimes(-)$ of the functor $t(\bbu)\otimes'(-)$.
Then $(t\circ V)'$, the scalar extension of $t\circ V$, is a quasi-inverse of the functor $t(\bbu)\otimes'(-)$,
this may again be proved using Proposition \ref{prop:univpropscalexbilin}(b).
Finally, $F'\to\End_{\mcT'}(t(\bbu))$ is an isomorphism since $\bbu$ has endomorphism ring $F$ and $t$ is
$F'/F$-fully faithful.

(b): This statement is true by construction, since we have given $\mcT'$ a structure of tensor category
extending that of $\mcT$.
\end{proof}

\begin{prop}\label{prop:tensorgivestensor}
Let $\mcT$ be an $F$-finite abelian tensor category over $F$, $\mcT'$ an $F'$-linear abelian tensor category,
$V:\:\mcT\to\mcT'$ an $F$-linear right exact tensor functor. Then the
$F'$-linear functor $V':\:\mcT\otimes_FF'\to\mcT'$ induced by Theorem \ref{thm:univpropscalex} is a tensor functor.
\end{prop}
\begin{proof}
The proof is similar to the proof of Theorem \ref{thm:scalextoftensorcats},
using Proposition \ref{prop:univpropscalexbilin} and the precise definition of
tensor functors. We suppress it.
\end{proof}

\begin{rem}
It follows that $t:\:\mcT\to\mcT\otimes_FF'$ has a universal property with
respect to tensor categories, with right exact tensor functors replacing 
the right exact functors of Theorem \ref{thm:univpropscalex}.
\end{rem}

\subsection{The Influence of Duals}

\begin{dfn}
\begin{enumerate}
\item An object $X$ of an abelian tensor category is \emph{dualisable}
if there exists an object $X^\vee$ -- its \emph{dual} -- together
with homomorphisms $\delta:\:\bbu\to X\otimes X^\vee$ and
$\ev:\:X\otimes X^\vee\to\bbu$ such that the composite homomorphisms
$X\to X\otimes X^\vee\otimes X\to X$ and $X^\vee\to X^\vee\otimes X\otimes X^\vee\to X^\vee$
are equal to the respective identities. If $X$ is dualisable, then
so is $X^\vee$ and one has a canonical isomorphism $X\isom X^{\vee\vee}$.
\item An abelian tensor category is \emph{rigid} if all of its
objects are dualisable.
\item The \emph{dual} of a homomorphism $f:\:X\to Y$
in a rigid abelian tensor category is the unique homomorphism $f^\vee:\:Y^\vee\to X^\vee$
satisfying
\[\ev_Y\circ(\id_{Y^\vee}\otimes f)=\ev_X\circ(f^\vee\otimes\id_X):\:Y^\vee\otimes X\to\bbu.\]
\item The \emph{internal Hom} of two objects $X,Y$ of a rigid abelian tensor category
is the object $\iHom(X,Y):=X^\vee\otimes Y$.
\item A \emph{pre-Tannakian category over $F$} is an $F$-finite rigid abelian tensor category over $F$.
\item A subcategory $\mcS$ of a pre-Tannakian category $\mcT$ is a \emph{strictly full pre-Tannakian subcategory}
  if it is full and closed under direct sums, tensor products, duals and subquotients in $\mcT$.
\item Given a set $S$ of objects of a pre-Tannakian category $\mcT$, we let $\llkurv S\rrkurv_{\otimes}$
  denote the smallest strictly full pre-Tannakian subcategory of $\mcT$ containing $S$. We
  also set $\llkurv X\rrkurv_{\otimes}:=\llkurv \{X\}\rrkurv_{\otimes}$ for any object $X$ of $\mcT$.
\item A pre-Tannakian category $\mcT$ is \emph{finitely generated} if $\mcT=\llkurv X\rrkurv_{\otimes}$
  for some object $X\in\mcT$.
\end{enumerate}
\end{dfn}

\begin{prop}\label{prop:scalextofrigidtensorcats}
Let $\mcT$ be a pre-Tannakian category over $F$, and consider a field extension $F'/F$.
Then $\mcT\otimes_FF'$ is a pre-Tannakian category over $F'$.
\end{prop}
\begin{proof}
$\mcT\otimes_FF'$ carries the natural structure of abelian tensor category given
by Theorem \ref{thm:scalextoftensorcats}. We must show that it is rigid, so
let $\bX$ be an object of $\mcT\otimes_FF'$. By Lemma \ref{lem:presandcopres1}
there exists a presentation
\[t(X_1)\arrover{f}t(X_0)\to\bX\to 0\]
of $\bX$ with objects $X_0,X_1$ of $\mcT$. Since $\mcT$ is rigid,
the objects $X_i$ are dualisable. Since $t$ is a tensor functor,
so are the objects $t(X_i)$, with duals $t(X_i^\vee)$. But every
object of an abelian tensor category which possesses a presentation
by dualisable objects is dualisable. Namely, $\bX^\vee:=\ker(f^\vee)$
is a dual of $\bX=\coker(f)$.
\end{proof}

For pre-Tannakian categories, we obtain yet another universal property
of $t$ with respect to \emph{exact} tensor functors,
due to the following fact.

\begin{lem}\label{lem:rightexactisexactforrigids}
Let $V:\:\mcS\to\mcT$ be a tensor functor. Assume that $\mcS$ is rigid.
Then $V$ is exact if and only if it is right exact.
\end{lem}
\begin{proof}
Every tensor functor commutes with duals. Dualisation is an exact functor.
So if $0\to\bX'\to\bX\to\bX''$ is a left exact sequence in $\mcS$,
then its image under $V$ may be identified with the dual of the image
of its dual, which must therefore be left exact.
\end{proof}

We end this subsection with the following two observations.

\begin{lem}\label{lem:rigidsautofaithful}
Let $\mcS,\mcT$ be abelian tensor categories,
$V:\:\mcS\to\mcT$ an exact tensor functor, and assume
that $\mcS$ is rigid. If $\mcT\neq 0$, then $V$ is faithful.
\end{lem}
\begin{proof}
An exact functor is faithful if and only if it maps all
non-zero objects to non-zero objects. A dualisable object
$X\in\mcS$ is non-zero if and only if $X\otimes X^\vee\to\bbu$
is surjective, and this criterion is respected by
right exact tensor functors. So if $\mcT\neq 0$, that is,
if $\bbu_{\mcT}\not\isom 0$, then $V$ is faithful.
\end{proof}

\begin{lem}\label{lem:aslfdkjsal}
Let $\mcT$ be a pre-Tannakian category over $F$, 
$\mcT'$ an abelian tensor category over $F'$, and
consider two exact $F$-linear tensor functors $V',W':\:\mcT\otimes_FF'\to\mcT'$
Let $\eta:\:V'\Rightarrow W'$ be a natural transformation.
Then $\eta$ is a morphism of tensor functors if and only if
its restriction $V\Rightarrow W$ along $t$ is such.
\end{lem}
\begin{proof}
Again, as in Theorem \ref{thm:scalextoftensorcats},
this is a matter of checking that certain natural transformations
are equal, and we suppress it.
\end{proof}

\subsection{Permanence of Relative Full Faithfulness}\label{ss:permofrelff}

\begin{prop}\label{prop:permofrelfulfaith}
Let $\mcA$ be a pre-Tannakian category over $F$, $\mcB$ an $F'$-linear abelian
tensor category, $V:\:\mcA\to\mcB$ an exact $F$-linear tensor functor and
$V':\:\mcA\otimes_FF'\to\mcB$ the induced exact $F'$-linear functor.
Then $V$ is $F'/F$-fully faithful if and only if $V'$ is fully faithful.
\end{prop}
\begin{proof}
If $V'$ is fully faithful, then its restriction $V=V'\circ t$ is $F'/F$-fully faithful
since $t$ is $F'/F$-fully faithful by Lemma \ref{lem:tisFpFff}.

Conversely, let us assume that $V$ is $F'/F$-fully faithful. We first prove
that for every $\bX\in\mcA\otimes_FF'$ and every $Y\in\mcA$, the homomorphism
\[V': \Hom_{\mcA'}\big(\bX,t(Y)\big)\To\Hom_{\mcB}\big(V'(\bX),V(Y)\big)\]
is an isomorphism. With the help of Lemma \ref{lem:presandcopres1} we choose a presentation
\begin{equation}\label{eqn:someprese}
t(X_1)\to t(X_0)\to\bX\to 0
\end{equation}of $\bX$. Applying $\Hom(-,t(Y))$ to this sequence, and applying $\Hom(-,V(Y))$ to the right
exact sequence which is the image of (\ref{eqn:someprese}) under $V'$, we obtain a commutative diagram with exact rows:
\[\def\objectstyle{\scriptstyle}
\def\labelstyle{\scriptstyle}
\xymatrix{
0 \ar[r] & \Hom\big(\bX,t(Y)\big) \ar[d]\ar[r] & \Hom\big(t(X_0),t(Y)\big) \ar[r]\ar[d] & \Hom\big(t(X_1),t(Y\big) \ar[d]\\
0 \ar[r] & \Hom\big(V'(\bX),V(Y)\big)         \ar[r]       & \Hom\big(V(X_0),V(Y)\big) \ar[r]                          & \Hom\big(V(X_1),V(Y)\big)
}\]
The two last vertical arrows are isomorphisms since both $t$ and $V$ are $F'/F$-fully faithful functors.
By the Five Lemma, the first vertical arrow is an isomorphism, as claimed.

In general, consider $\bX$ and $\bY$ in $\mcA\otimes_FF'$. The dual of a presentation of $\bY^\vee$ gives
us a copresentation
\begin{equation}\label{eqn:someprese2}
0\to\bY\to t(Y^0)\to t(Y^1)
\end{equation} of $\bY$. Applying $\Hom(\bX,-)$ to this sequence, and
applying $\Hom(-,V'\bY)$ to the left exact sequence which is the image of (\ref{eqn:someprese2}) under $V'$, we obtain a diagram
\[\def\objectstyle{\scriptstyle}
\def\labelstyle{\scriptstyle}
\xymatrix{
0 \ar[r] & \Hom\big(\bX,\bY\big)     \ar[d]\ar[r] & \Hom\big(\bX,t(Y^0)\big) \ar[r]\ar[d] & \Hom\big(\bX,t(Y^1)\big) \ar[d]\\
0 \ar[r] & \Hom\big(V'(\bX),V'(\bY)\big)       \ar[r] & \Hom\big(V'(\bX),V(Y^0)\big) \ar[r]       & \Hom\big(V'(\bX),V(Y^1)\big)
}\]
By what we have already proven, the last two vertical arrows are isomorphisms, so by the Five Lemma
so is the first, and we have shown that $V'$ is fully faithful.

\end{proof}

\subsection{Induced Equivalences}\label{ss:inducedequivalences}

\begin{thm}\label{thm:mainthm2}
Let $F'/F$ be a separable field extension, $\mcT$ a pre-Tannakian category
over $F$, $\mcT'$ a pre-Tannakian category over $F'$ and consider an
$F$-linear exact tensor functor $V:\:\mcT\to\mcT'$. Let $\llkurv V\mcT\rrkurv_{\otimes}$
denote the strictly full pre-Tannakian subcategory of $\mcT'$ generated by the essential
image of $V$.

If $V$ is $F'/F$-fully faithful and semisimple on objects, then the
functor
\[V':\:\mcT\otimes_FF'\To\llkurv V\mcT\rrkurv_{\otimes}\]
induced by $V$ is an equivalence of pre-Tannakian categories.
\end{thm}
\begin{proof}
The functor $V'$ is an $F'$-linear exact tensor functor
by Theorem \ref{thm:univpropscalex}(a), Proposition \ref{prop:tensorgivestensor} and
Lemma \ref{lem:rightexactisexactforrigids}.
It is fully faithful by Proposition \ref{prop:permofrelfulfaith}.
It is essentially surjective
by Propositions \ref{prop:permofsemisimpli} and
\ref{prop:ffssisessimsubquotclosed}.
Therefore, it is an equivalence of pre-Tannakian categories. 
\end{proof}

\section{Tannakian Categories}

\subsection{Scalar Extension of Tannakian Categories}

Let $F$ be a field.

\begin{dfn}
\begin{enumerate}
  \item Let $R$ be a commutative $F$-algebra. A \emph{fibre functor over $R$} of a pre-Tannakian category $\mcT$ over $F$
  is a faithful $F$-linear exact tensor functor $\omega$ from $\mcT$ to the category
  of $R$-modules which has values in the rigid subcategory 
  of finitely generated projective $R$-modules.
  \item A \emph{neutral}
  fibre functor is a fibre functor over $F$ itself, it takes values in the
  category $\Vect_F$ of finite-dimensional $F$-vector spaces.
  \item A \emph{Tannakian category over $F$} is a pre-Tannakian category for which there
  exists a fibre functor over some field extension $F'/F$. If there exists a neutral fibre functor,
  we say that $\mcT$ is \emph{neutral}.
  \item A subcategory $\mcS$ of a Tannakian category $\mcT$ is a \emph{strictly full Tannakian subcategory}
  if it is a strictly full pre-Tannakian subcategory of $\mcT$, that is,
  if it is full and closed under direct sums, tensor products, duals and subquotients in $\mcT$.
\end{enumerate}
\end{dfn}

We start by checking that our notion of scalar extension
for abelian categories gives rise to a notion of scalar extension for
Tannakian categories. In particular, Tannakian categories may
be ``neutralised''. Together with the following Theorem \ref{thm:nonneutralTannaka},
we generalise \cite[Proposition 3.11]{DeM82} and substantiate
\cite[Proposition A.12]{Mil92}.

\begin{prop}\label{prop:equivalenceoffibrefunctors}
Let $\mcT$ be a pre-Tannakian category over $F$, and consider a field
extension $F'/F$. For every commutative $F'$-algebra $R'$ the restriction
functor
\[\left(\!\!\!\left(\text{\begin{tabular}{c}fibre functors on\\$\mcT\otimes_FF'$ over $R'$\end{tabular}}\right)\!\!\!\right) \arrover{(-)\circ t} %
  \left(\!\!\!\left(\text{\begin{tabular}{c}fibre functors on\\$\mcT$ over $R'$\end{tabular}}\right)\!\!\!\right)\]
is an equivalence of categories.
\end{prop}
\begin{proof}
The given functor $\res:=(-)\circ t$ maps fibre functors on
$\mcT\otimes_FF'$ to fibre functors on $\mcT$ since $t$
is exact, $F$-linear, faithful by either
Proposition \ref{lem:tisFpFff} or Lemma \ref{lem:rigidsautofaithful},
and a tensor functor by Theorem \ref{thm:scalextoftensorcats}(b).
Hence, $\res$ is well-defined. It is fully faithful by
Proposition \ref{propunivpropscalexfuncinV} and Lemma
\ref{lem:aslfdkjsal}.

To show that $\res$ is essentially surjective,
let $\omega$ be a fibre functor on $\mcT$ over a given
$F'$-algebra $R'$. The $F$-linear right-exact
functor $\omega'$ on $\mcT\otimes_FF'$ induced by Theorem
\ref{thm:univpropscalex} fulfills $\res(\omega')\isom\omega$
by item (a) of that theorem. Now $\omega'$ is exact by
Lemma \ref{lem:rightexactisexactforrigids}, faithful by
Lemma \ref{lem:rigidsautofaithful}, and a tensor functor
by Proposition \ref{prop:tensorgivestensor}.
A priori, $\omega'$ has values in the category of $R'$-modules.
However, since $\mcT\otimes_FF'$ is rigid by Proposition
\ref{prop:scalextofrigidtensorcats}, the essential
image of $\omega'$ must consist of dualisable $R'$-modules
(cf.\ the proof of Proposition \ref{prop:scalextofrigidtensorcats}).
It is well known that a dualisable $R'$-module is finitely
generated and projective, see \cite{Del87}.
Hence, $\omega$ is a fibre functor on
$\mcT\otimes_FF'$ over $R'$, and we are done.
\end{proof}

\begin{thm}\label{thm:scalextoftannakiancats}
Let $\mcT$ be a Tannakian category over $F$, and consider
a field extension $F'/F$.
\begin{enumerate}
  \item $\mcT\otimes_FF'$ is a Tannakian category over $F'$.
  \item If $\mcT$ has a fibre functor over $F'$, then $\mcT\otimes_FF'$ is neutral.
\end{enumerate}
\end{thm}
\begin{proof}
(a): By Proposition \ref{prop:scalextofrigidtensorcats} we know that $\mcT\otimes_FF'$
is a pre-Tannakian category over $F'$.
By assumption, there exists a fibre functor
of $\mcT$ over some field extension $L/F$. Choose a field extension
$L'/F$ containing both $F'$ and $L$. Then $(L'\otimes_{L}-)\circ\omega$
is a fibre functor of $\mcT$ over $L'$. By Proposition
\ref {prop:equivalenceoffibrefunctors}, it extends to a fibre
functor of $\mcT\otimes_FF'$ over $L'$.

(b): In this case, we may choose $L'=L=F'$.
\end{proof}

The starting point of Tannakian duality is the
idea that the category of finite-dimensional representations
of a linear algebraic group is in a certain sense
dual to the group itself. Reversing this point of
view, we wish to associate a group to a Tannakian
category.

\begin{dfn}\begin{enumerate}
\item The \emph{algebraic monodromy group} of the Tannakian category $\mcT$ over $F$ with respect
to a given fibre functor $\omega$ over a field extension $F'/F$ is the functor
\[G_\omega(\mcT):\:\llkurv\text{Commutative $F'$-Algebras}\rrkurv\To \llkurv\text{Groups}\rrkurv\]
mapping a commutative $F'$-algebra $R'$ to the group $\Aut^\otimes\big(R'\otimes_{F'}\omega(-)\big)$
of tensor automorphisms of the functor $R'\otimes_{F'}\omega(-)$ which maps 
an object $X$ of $\mcT$ to the $R'$-module $R'\otimes_{F'}\omega(X)$.

\item The \emph{algebraic monodromy group} $G_\omega(X)$ of an object $X$ of $\mcT$ with respect to $\omega$
is the algebraic monodromy group of the strictly full Tannakian subcategory $\llkurv X\rrkurv_{\otimes}$
of $\mcT$ that $X$ generates, with respect to the restriction of $\omega$.
\end{enumerate}\end{dfn}

The theory of Tannakian categories comes in two flavours, neutral and non-neutral.
The former is relatively simple to understand, whereas the latter is more advanced
and more closely connected to groupoids than groups.
The non-neutral theory is developed in \cite{Del90}. Nevertheless, the aim of
this section is to understand part of the non-neutral theory, using only our
results on scalar extension and the neutral theory which we recall
in the next two theorems.

\begin{thm}\label{thm:reconstructgroup}
Let $G$ be an algebraic group over $F$. Then $G$ represents the monodromy group
of $\Rep_F(G)$ with respect to the forgetful functor $\Rep_F(G)\to\Vect_{F}$.
\end{thm}
\begin{proof}
\cite[Theorem 2.8]{DeM82}.
\end{proof}

\begin{thm}\label{thm:neutralTannaka}
Let $\mcT$ be a neutral Tannakian category over $F$, and fix a neutral fibre functor $\omega$.
\begin{enumerate}
  \item $G_\omega(\mcT)$ is representable by an affine group scheme over $F$.
  \item $G_\omega(\mcT)$ is of finite type if and only if $\mcT$ is finitely generated.
  \item If $\mcT$ is finitely generated, then $\omega(X)$ is a faithful representation
  of $G_\omega(\mcT)$ for every $X\in\mcT$ with $\mcT=\llkurv X\rrkurv_\otimes$.
  \item $\omega$ induces an equivalence of categories $\mcT\To\Rep_F(G_\omega(\mcT))$.
\end{enumerate}
\end{thm}
\begin{proof}
\cite{Saa72} or \cite[Theorem 2.11]{DeM82}.
\end{proof}

We end this subsection with a version of Theorem \ref{thm:neutralTannaka}
for non-neutral Tannakian categories and a consequence of
Theorem \ref{thm:mainthm2}.

\begin{thm}\label{thm:nonneutralTannaka}
Let $\mcT$ be a Tannakian category over $F$, and fix a fibre functor
$\omega$ over a field extension $F'/F$.
\begin{enumerate}
  \item $G_\omega(\mcT)$ is an affine group scheme over $F'$.
  \item $G_\omega(\mcT)$ is of finite type if and only if $\mcT$ is finitely generated.
  \item If $\mcT$ is finitely generated, then $\omega(X)$ is a faithful representation
  of $G_\omega(\mcT)$ for every $X\in\mcT$ with $\mcT=\llkurv X\rrkurv_\otimes$.
  \item $\omega$ induces an equivalence of categories $\mcT\otimes_FF'\To\Rep_{F'}(G_\omega(\mcT))$.
\end{enumerate}
\end{thm}
\begin{proof}
By Theorem \ref{thm:scalextoftannakiancats}, $\mcT\otimes_FF'$
is a Tannakian category, and the functor $\omega'$ induced by $\omega$
is a neutral fibre functor. Therefore, Theorem \ref{thm:neutralTannaka}
applies to the pair $(\mcT\otimes_FF',\omega')$.

It remains to show that $G_\omega(\mcT)$ and $G_{\omega'}(\mcT\otimes_FF')$ coincide.
But given an $F'$-algebra $R'$, Proposition \ref{prop:equivalenceoffibrefunctors}
shows that the natural homomorphism
\[\xymatrix{
\Aut^\otimes\left((R'\otimes_{F'}-)\circ\omega'\right) \ar[r] & \Aut^\otimes\left((R'\otimes_{F'}-)\circ\omega\right)\\
G_{\omega'}(\mcT\otimes_FF')(R') \ar@{=}[u] & G_\omega(\mcT)(R')\ar@{=}[u]}\]
a bijection, so we are done.
\end{proof}

\begin{prop}\label{cor:whengroupsarethesame}
Let $F'/F$ be a separable field extension, $F''/F'$ any
field extension, $\mcT$ a Tannakian category
over $F$, $\mcT'$ a Tannakian category over $F'$ with
fibre functor $\omega$ over $F''$ and consider an
$F$-linear exact tensor functor $V:\:\mcT\to\mcT'$. Assume that
$V$ is $F'/F$-fully faithful and semisimple on objects.

For every object $X$ of $\mcT$, there exists a canonical isomorphism of algebraic
monodromy groups \[G_{\omega\circ V}(X)\larrover{\isom}G_\omega\big(V(X)\big).\]
\end{prop}
\begin{proof}
By Theorems \ref{thm:nonneutralTannaka}(d) and \ref{thm:reconstructgroup},
the monodromy group $G_{\omega\circ V}(X)$ coincides with the
monodromy group of $t(X)$ as calculated in $\mcT\otimes_FF'$ with respect
to $(\omega\circ V)'$:
\[G_{\omega\circ V}(X)\larrover{\isom}G_{(\omega\circ V)'}(t(X)).\]
Applying Theorem \ref{thm:mainthm2} to the Tannakian categories
$\llkurv X\rrkurv_\otimes$ and $\llkurv V(X)\rrkurv_\otimes$, we obtain
an equivalence of categories
\[\llkurv X\rrkurv_\otimes\otimes_FF'\arrover{\isom}\llkurv V(X)\rrkurv_\otimes.\]
Clearly, this implies the existence of an isomorphism
\[G_{(\omega\circ V)'}(t(X))\larrover{\isom}G_\omega(V(X)),\]
which is what was left to prove.
\end{proof}

\subsection{Reductivity of Monodromy Groups}

Let $F$ be a field.

\begin{prop}\label{thm:redcrit}
Let $V$ be a finite-dimensional $F$-vector space, and consider a closed algebraic subgroup $G\subset\GL(V)$.
If $V$ is semisimple as a representation of $G$, and $\End_G(V)$ is a separable $F$-algebra, then the identity
component $G^\circ$ is a reductive group.
\end{prop}
\begin{proof}
Let $\overline{F}$ be an algebraic closure of $F$.
Since $\End_G(V)$ is both semisimple and separable over $F$, the
$\Falg$-algebra $\Falg\otimes_FE$ is semisimple by Proposition \ref{prop:separablealgebras}(b).
By the same assumptions, $\Falg\otimes_FV$ is a semisimple
representation of $G_{\overline{F}}$, the base change of $G$ to $\overline{F}$,
using Proposition \ref{prop:semisimpllll} applied to Example \ref{ex:scalarextcats}(b).
Therefore we may assume that $F$ is algebraically closed.

Let $U$ be the unipotent radical of $G$, and let $V^U\subset V$ denote
the sub-vector space consisting of those elements fixed (pointwise) by $U$.
Since $U$ is normal in $G$, $V^U$ is a $G$-stable subspace of $V$.
We claim that $V^U=V$. If not, since $V$ is semisimple, we may write $V=V^U\oplus V'$ for some
$G$-stable complement $V'$ of $V^U$. Since $U$ operates unipotently
on $V'$, it follows that $(V')^U\neq 0$, which is a contradiction
to the definition of $V'$ as a complement of $V^U$. Therefore
$V^U=V$. Since $G$ operates faithfully on $V$, it follows that $U=1$,
which means that $G^\circ$ is reductive.
\end{proof}

\begin{cor}
Let $\mcT$ be a Tannakian category over $F$, fix a fibre functor $\omega$ over some
field extension $F'/F$, and choose an object $X$ of $\mcT$. If $X$ is
semisimple and $\End(X)$ is a separable $F$-algebra, then the
identity component of $G_\omega(X)$ is a reductive group over $F'$.
\end{cor}
\begin{proof}
The vector space $\omega(X)$ is a faithful representation of
$G_\omega(X)$ by Proposition \ref{thm:nonneutralTannaka}(c).
Therefore, Proposition \ref{thm:redcrit} applies to it, and we are done.
\end{proof}

\subsection{An Application: Representation-Valued Fibre Functors}

We close this article with an application of
our results to ``representation-valued'' fibre functors.
Let $\Gamma$ be a profinite group. 
Let $F$ be a global field, $F'\supset F$ a local field
arising by completing $F$ at some place. It is well-known
that the field extension $F'/F$ is separable.
Let $\mcT$ be any Tannakian category over $F$,
and let $\Rep_{F'}\Gamma$ denote the category of
finite-dimensional continuous representations of $\Gamma$ over $F'$.

We assume that we are given a faithful exact $F$-linear tensor functor
\[V:\:\mcT\To\Rep_{F'}\Gamma,\]
a ``representation-valued fibre functor'', which is \emph{additionally}
both $F'/F$-fully faithful and semisimple on objects. Examples
are given by the rational Tate module functors on either the Tannakian
category of pure Grothendieck motives generated by abelian varieties up to isogeny
or the Tannakian category of Anderson $A$-motives up to isogeny.

For every object $X$ of $\mcT$, let $\Gamma(X)$ denote the image of $\Gamma$
in $\Aut_{F'}(V(X))$, and let $G(X)$ denote the algebraic monodromy group of $X$
with respect to the fibre functor on $\mcT$ arising by postcomposing $V$ with the forgetful functor
$U:\:\Rep_{F'}\Gamma\to\Vect_{F'}$.

There exists a unique reduced algebraic subgroup of $\GL(V(X))$ which has as set of $F'$-rational
points the Zariski closure of $\Gamma$ in $\GL(V(X))(F')$, and it is natural to hope that this
group coincides with $G(X)$:

\begin{thm}\label{thm:zdenseredabstr}
\begin{enumerate}
  \item The natural homomorphism $\Gamma(X)\to G(X)(F')$ is injective
  and has Zariski-dense image.
  \item If $X$ is semisimple and $\End_{\mcT}(X)$ is a separable $F$-algebra, then $G(X)^\circ$, the identity
    component of $G(X)$, is a reductive group.
\end{enumerate}
\end{thm}

We need some preparations.

\begin{lem}\label{lem:ffandsimplicityofalg}
Let $V$ be a finite-dimensional $F'$-vector space, and consider an algebraic subgroup
$G\subset\GL(V)$ together with a Zariski-dense subgroup $\Gamma\subset G(F')$ of its
$F'$-rational points. Then:
\begin{enumerate}
  \item A linear subspace $W\subset V$ is $G$-stable if and only if it is $\Gamma$-stable.
  \item We have $\End_G(V)=\End_{\Gamma}(V)$.
\end{enumerate}
\end{lem}
\begin{proof}
(a): Given a linear subspace $W\subset V$ the stabiliser $H:=\Stab_G(W)$ is
an algebraic subgroup of $G$. If $W$ is $G$-stable, then the $F'$-valued points of $H=G$
contain $\Gamma$, so $W$ is $\Gamma$-stable.

Conversely, if $W$ is $\Gamma$-stable, then $H(F')$ contains $\Gamma$. Since
$\Gamma$ is dense in $G(F')$, this implies that $H=G$, and so $W$ is $G$-stable.

(b): We note that $\End_G(V)$ is the maximal $G$-stable subspace of $V^\vee\otimes V$
on which $G$ acts trivially, and similarly $\End_\Gamma(V)$ is the maximal $\Gamma$-stable subspace
on which $\Gamma$ acts trivially. By a similar argument as in (a), these two spaces must coincide.
\end{proof}

\begin{prop}\label{prop:mondrofcontreps}
Let $V$ be a finite-dimensional $F'$-vector space, consider a subgroup $\Gamma\subset\GL(V)(F')$
with associated algebraic group $G:=\overline{\Gamma}^{Zar}\subset\GL(V)$. Let $V^\mathrm{cont}$ represent
$V$ considered as a continuous representation of $\Gamma$ over $F'$, and let $V^\mathrm{alg}$ represent
$V$ considered as a representation of $G$ over $F'$.
\begin{enumerate}
  \item The natural functor
          \[\llkurv V^\mathrm{alg}\rrkurv_\otimes\To\llkurv V^\mathrm{cont}\rrkurv_\otimes\]
        between the strictly full Tannakian subcategories of $\Rep_{F'}G$ and of $\Rep_{F'}\Gamma$
        generated by $V^\mathrm{alg}$ and $V^\mathrm{cont}$, respectively, is an equivalence of Tannakian categories.
  \item In particular, $G$ is the the algebraic monodromy group of $V^\mathrm{cont}$.
\end{enumerate}
\end{prop}
\begin{proof}
(a): Any object of $\llkurv V^\mathrm{alg}\rrkurv_\otimes$ yields
a continuous representation of $\Gamma$, and this gives rise to the desired exact $F'$-linear
tensor functor; let us denote it by $C$. We wish to employ Theorem \ref{thm:mainthm2} to conclude that $C$
is an equivalence of Tannakian categories, so we must show that $C$ is fully faithful and semisimple,
let us do this.

Consider $W\in\llkurv V^\mathrm{alg}\rrkurv_\otimes$, let $G_W$ denote the image of $G$ in $\GL(W)$ and
let $\Gamma_W$ denote the image of $\Gamma$ in $G_W(F')$. By continuity, $\Gamma_W$ is dense in
$G_W(F')$, so Lemma \ref{lem:ffandsimplicityofalg}(b) shows that $\End_G(W)=\End_\Gamma(CW)$.
Since this is true for all $W$, we conclude that $C$ is fully faithful. If $W$ is simple,
Lemma \ref{lem:ffandsimplicityofalg}(a) shows that $CW$ is simple. In
particular, $C$ is semisimple on objects.

(b): It is well-known (cf. \cite[Theorem 3.5]{WCW}) that $\llkurv V^\mathrm{alg}\rrkurv_\otimes$ is equivalent
to $\Rep_{F'}(G)$. Thus, by Theorem \ref{thm:reconstructgroup},
$G$ is the algebraic monodromy group of $V^\mathrm{alg}$, and so by (a) $G$ is also
the algebraic monodromy group of $V^\mathrm{cont}$.
\end{proof}

\begin{proof}[Proof of Theorem \ref{thm:zdenseredabstr}]
(a): By Corollary \ref{cor:whengroupsarethesame} we have
\[G_U(V(X))\isom G_{U\circ V}(X).\]
By Proposition \ref{prop:mondrofcontreps}, $\Gamma(X)\subset G_{U\circ V}(X)(F')$
is Zariski dense.

(b): By our assumptions or Theorem \ref{thm:nonneutralTannaka}(c), $G(X)$ is a closed algebraic subgroup of $\GL(V(X))$, and $V(X)$
is semisimple as a representation of $G(X)$, since $X$ is semisimple
and $V$ is semisimple on objects. Since
$V$ is $F'/F$-fully faithful, $\End(V(X))=F'\otimes_F\End(X)$, which is a separable $F'$-algebra
since $\End(X)$ is a separable $F$-algebra. Therefore, the assumptions of Theorem \ref{thm:redcrit}
hold true, and $G(X)^\circ$ is a reductive group.
\end{proof}

\bibliographystyle{alpha}

\end{document}